\documentclass[12pt]{amsart}
\usepackage{txfonts}      
\usepackage{amssymb}
\usepackage[mathscr]{eucal}
\usepackage{amsmath}
\usepackage{amscd}
\usepackage[dvips]{color}
\usepackage{multicol}
\usepackage[all]{xy}           
\usepackage{graphicx}
\usepackage{color}
\usepackage{colordvi}
\usepackage{xspace}
\usepackage[colorlinks,final,backref=page,hyperindex,hypertex]{hyperref}

\usepackage{tikz}

\begin{document}

\newcommand{\nc}{\newcommand}
\newcommand{\delete}[1]{}
\nc{\mfootnote}[1]{\footnote{#1}} 
\nc{\todo}[1]{\tred{To do:} #1}

\nc{\mlabel}[1]{\label{#1}}  
\nc{\mcite}[1]{\cite{#1}}  
\nc{\mref}[1]{\ref{#1}}  
\nc{\mbibitem}[1]{\bibitem{#1}} 

\delete{
\nc{\mlabel}[1]{\label{#1}  
{\hfill \hspace{1cm}{\bf{{\ }\hfill(#1)}}}}
\nc{\mcite}[1]{\cite{#1}{{\bf{{\ }(#1)}}}}  
\nc{\mref}[1]{\ref{#1}{{\bf{{\ }(#1)}}}}  
\nc{\mbibitem}[1]{\bibitem[\bf #1]{#1}} 
}

\newtheorem{theorem}{Theorem}[section]
\newtheorem{prop}[theorem]{Proposition}
\newtheorem{defn}[theorem]{Definition}
\newtheorem{lemma}[theorem]{Lemma}
\newtheorem{coro}[theorem]{Corollary}
\newtheorem{prop-def}[theorem]{Proposition-Definition}
\newtheorem{claim}{Claim}[section]
\newtheorem{remark}[theorem]{Remark}
\newtheorem{propprop}{Proposed Proposition}[section]
\newtheorem{conjecture}{Conjecture}
\newtheorem{exam}[theorem]{Example}
\newtheorem{assumption}{Assumption}
\newtheorem{condition}[theorem]{Assumption}
\newtheorem{question}[theorem]{Question}

\renewcommand{\labelenumi}{{\rm(\alph{enumi})}}
\renewcommand{\theenumi}{\alph{enumi}}

\nc{\tred}[1]{\textcolor{red}{#1}}
\nc{\tblue}[1]{\textcolor{blue}{#1}}
\nc{\tgreen}[1]{\textcolor{green}{#1}}
\nc{\tpurple}[1]{\textcolor{purple}{#1}}
\nc{\btred}[1]{\textcolor{red}{\bf #1}}
\nc{\btblue}[1]{\textcolor{blue}{\bf #1}}
\nc{\btgreen}[1]{\textcolor{green}{\bf #1}}
\nc{\btpurple}[1]{\textcolor{purple}{\bf #1}}

\nc{\opone}{\quad \begin{picture}(5,5)(0,0)
\vector(-1,1){10}
\end{picture}}

\nc{\optwo}{\begin{picture}(5,5)(0,0)
\vector(0,1){10}
\end{picture}}

\nc{\opthree}{\begin{picture}(5,5)(0,0)
\vector(1,1){10}
\end{picture}\,\,}

\nc{\opfour}{\quad \begin{picture}(5,5)(0,0)
\vector(-1,1){10}\vector(0,1){10}
\end{picture}}

\nc{\opfive}{\quad \begin{picture}(5,5)(0,0)
\vector(-1,1){10}\vector(1,1){10}
\end{picture}\,\,}

\nc{\opsix}{ \quad \begin{picture}(5,5)(0,0)
\vector(0,1){10}\vector(1,1){10}
\end{picture}\,\,}

\nc{\opseven}{\quad \begin{picture}(5,5)(0,0)
\vector(-1,1){10}\vector(0,1){10}\vector(1,1){10}
\end{picture}\,\,}

\nc{\postone}{\{\cdot,\cdot,\cdot\}}
\nc{\postthree}{[\cdot,\cdot,\cdot]}
\nc{\posttwo} {(\cdot,\cdot,\cdot)}

\nc{\graft} {\omega (\tau_{1} \vee \tau_{2} \vee \cdots \vee \tau_{\ell})}

\nc{\conf}{configuration\xspace}
\nc{\confs}{configurations\xspace}
\nc{\confm}{\mathcal{C}}
\nc{\confn}{C}

\nc{\vertset}{\Omega} 
\nc{\leafset}{\calx} 
\nc{\gensp}{V} 
\nc{\relsp}{R} 
\nc{\leafsp}{\mathcal{X}}    
\nc{\treesp}{\mathbb{T}} 
\nc{\genbas}{\mathcal{V}} 
\nc{\opd}{\mathcal{P}} 

\nc{\vin}{{\mathrm Vin}}    
\nc{\lin}{{\mathrm Lin}}    
\nc{\inv}{\mathrm{I}n}

\nc{\bvp}{V_P}     

\nc{\gop}{{\,\omega\,}}     
\nc{\gopb}{{\,\nu\,}}
\nc{\svec}[2]{{\textrm{\tiny{$\left(\begin{matrix}#1\\
#2\end{matrix}\right)$}}}}  
\nc{\ssvec}[2]{{\textrm{\tiny{$\left(\begin{matrix}#1\\
#2\end{matrix}\right)$}}}} 
\nc{\treeg}[5]{\vcenter{\xymatrix@M=1.5pt@R=1.5pt@C=0pt{#1 & & #2 & & & #3 \\ & #5 \ar@{-}[lu] \ar@{-}[ru] & & & & \\ & & #4 \ar@{-}[lu] \ar@{-}[rrruu] & & \\ & & & & & \\ & & \ar@{-}[uu] & & & }}}
\nc{\treed}[5]{\vcenter{\xymatrix@M=2pt@R=4pt@C=2pt{#1 & & & #2  & & #3 \\ & & & & #5 \ar@{-}[lu] \ar@{-}[ru] & \\ & & & #4 \ar@{-}[ru] \ar@{-}[llluu] & \\ & & & & & \\ & & & \ar@{-}[uu] & & }}}

\nc{\tsvec}[3]{{\textrm{\tiny{$\left(\begin{matrix}#1\\
#2\\#3\end{matrix}\right)$}}}}  

\nc{\stsvec}[3]{{\textrm{\tiny{$\left(\begin{matrix}#1\\
#2\\#3\end{matrix}\right)$}}}} 

\nc{\ssp}{\calc\mathrm{Sp}}
\nc{\spl}{$\calc$}
\nc{\su}{\mathrm{Su}}
\nc{\dsu}{\cala\mathrm{Sp}}
\nc{\tsu}{\calb\mathrm{Sp}}
\nc{\TSu}{\mathrm{TSu}}
\nc{\srb}{\calc\mathrm{RB}}
\nc{\eval}[1]{{#1}_{\big|D}}
\nc{\oto}{\leftrightarrow}

\nc{\du}{\mathrm{Du}}
\nc{\tdu}{\mathrm{Tri}}
\nc{\rep}{\dagger}

\nc{\oaset}{\mathbf{O}^{\rm alg}}
\nc{\omset}{\mathbf{O}^{\rm mod}}
\nc{\oamap}{\Phi^{\rm alg}}
\nc{\ommap}{\Phi^{\rm mod}}
\nc{\ioaset}{\mathbf{IO}^{\rm alg}}
\nc{\iomset}{\mathbf{IO}^{\rm mod}}
\nc{\ioamap}{\Psi^{\rm alg}}
\nc{\iommap}{\Psi^{\rm mod}}

\nc{\suc}{{successor}\xspace}
\nc{\tsuc}{{trisuccessor}\xspace}
\nc{\Suc}{{Successor}\xspace}
\nc{\sucs}{{successors}\xspace}
\nc{\Sucs}{{Successors}\xspace}
\nc{\Tsuc}{{T-successor}\xspace}
\nc{\Tsucs}{{T-successors}\xspace}

\nc{\ddup}{{duplicator}\xspace}
\nc{\Ddup}{{Duplicator}\xspace}
\nc{\ddups}{{duplicators}\xspace}
\nc{\Ddups}{{Duplicators}\xspace}
\nc{\tdup}{{triplicator}\xspace}
\nc{\tdups}{{triduplicators}\xspace}
\nc{\Tdup}{{Triplicator}\xspace}
\nc{\Tdups}{{Triplicators}\xspace}

\nc{\perm}{{perm}}
\nc{\tperm}{{tri-perm}}

\nc{\Lsuc}{{L-successor}\xspace}
\nc{\Lsucs}{{L-successors}\xspace} \nc{\Rsuc}{{R-successor}\xspace}
\nc{\Rsucs}{{R-successors}\xspace}

\nc{\bia}{{$\mathcal{P}$-bimodule ${\bf k}$-algebra}\xspace}
\nc{\bias}{{$\mathcal{P}$-bimodule ${\bf k}$-algebras}\xspace}

\nc{\rmi}{{\mathrm{I}}}
\nc{\rmii}{{\mathrm{II}}}
\nc{\rmiii}{{\mathrm{III}}}

\nc{\pll}{\beta}
\nc{\plc}{\epsilon}

\nc{\ass}{{\mathit{Ass}}}
\nc{\comm}{{\mathit{Comm}}}
\nc{\comtrias}{{\mathit{ComTrias}}}
\nc{\TriLeib}{{\mathit{TriLeib}}}
\nc{\dend}{{\mathit{Dend}}}
\nc{\dias}{{\mathit{Dias}}}
\nc{\leib}{{\mathit{Leib}}}
\nc{\ldend}{{\mathit{LDend}}}
\nc{\lie}{{\mathit{Lie}}}
\nc{\lquad}{{\mathit{LQuad}}}
\nc{\octo}{{\mathit{Octo}}}
\nc{\Perm}{{\mathit{Perm}}}
\nc{\postlie}{{\mathit{PostLie}}}
\nc{\prelie}{{\mathit{preLie}}}
\nc{\Prelie}{{\mathit{PreLie}}}
\nc{\quado}{{\mathit{Quad}}}
\nc{\tridend}{{\mathit{TriDend}}}
\nc{\zinb}{{\mathit{Zinb}}}

 \nc{\adec}{\check{;}} \nc{\aop}{\alpha}
\nc{\dftimes}{\widetilde{\otimes}} \nc{\dfl}{\succ} \nc{\dfr}{\prec}
\nc{\dfc}{\circ} \nc{\dfb}{\bullet} \nc{\dft}{\star}
\nc{\dfcf}{{\mathbf k}} \nc{\apr}{\ast} \nc{\spr}{\cdot}
\nc{\twopr}{\circ} \nc{\tspr}{\star} \nc{\sempr}{\ast}
\nc{\disp}[1]{\displaystyle{#1}}
\nc{\bin}[2]{ (_{\stackrel{\scs{#1}}{\scs{#2}}})}  
\nc{\binc}[2]{ \left (\!\! \begin{array}{c} \scs{#1}\\
    \scs{#2} \end{array}\!\! \right )}  
\nc{\bincc}[2]{  \left ( {\scs{#1} \atop
    \vspace{-.5cm}\scs{#2}} \right )}  
\nc{\sarray}[2]{\begin{array}{c}#1 \vspace{.1cm}\\ \hline
    \vspace{-.35cm} \\ #2 \end{array}}
\nc{\bs}{\bar{S}} \nc{\dcup}{\stackrel{\bullet}{\cup}}
\nc{\dbigcup}{\stackrel{\bullet}{\bigcup}} \nc{\etree}{\big |}
\nc{\la}{\longrightarrow} \nc{\fe}{\'{e}} \nc{\rar}{\rightarrow}
\nc{\dar}{\downarrow} \nc{\dap}[1]{\downarrow
\rlap{$\scriptstyle{#1}$}} \nc{\uap}[1]{\uparrow
\rlap{$\scriptstyle{#1}$}} \nc{\defeq}{\stackrel{\rm def}{=}}
\nc{\dis}[1]{\displaystyle{#1}} \nc{\dotcup}{\,
\displaystyle{\bigcup^\bullet}\ } \nc{\sdotcup}{\tiny{
\displaystyle{\bigcup^\bullet}\ }} \nc{\hcm}{\ \hat{,}\ }
\nc{\hcirc}{\hat{\circ}} \nc{\hts}{\hat{\shpr}}
\nc{\lts}{\stackrel{\leftarrow}{\shpr}}
\nc{\rts}{\stackrel{\rightarrow}{\shpr}} \nc{\lleft}{[}
\nc{\lright}{]} \nc{\uni}[1]{\widetilde{#1}} \nc{\wor}[1]{\check{#1}}
\nc{\free}[1]{\bar{#1}} \nc{\den}[1]{\check{#1}} \nc{\lrpa}{\wr}
\nc{\curlyl}{\left \{ \begin{array}{c} {} \\ {} \end{array}
    \right .  \!\!\!\!\!\!\!}
\nc{\curlyr}{ \!\!\!\!\!\!\!
    \left . \begin{array}{c} {} \\ {} \end{array}
    \right \} }
\nc{\leaf}{\ell}       
\nc{\longmid}{\left | \begin{array}{c} {} \\ {} \end{array}
    \right . \!\!\!\!\!\!\!}
\nc{\ot}{\otimes} \nc{\sot}{{\scriptstyle{\ot}}}
\nc{\otm}{\overline{\ot}}
\nc{\ora}[1]{\stackrel{#1}{\rar}}
\nc{\ola}[1]{\stackrel{#1}{\la}}
\nc{\pltree}{\calt^\pl}
\nc{\epltree}{\calt^{\pl,\NC}}
\nc{\rbpltree}{\calt^r}
\nc{\scs}[1]{\scriptstyle{#1}} \nc{\mrm}[1]{{\rm #1}}
\nc{\dirlim}{\displaystyle{\lim_{\longrightarrow}}\,}
\nc{\invlim}{\displaystyle{\lim_{\longleftarrow}}\,}
\nc{\mvp}{\vspace{0.5cm}} \nc{\svp}{\vspace{2cm}}
\nc{\vp}{\vspace{8cm}} \nc{\proofbegin}{\noindent{\bf Proof: }}
\nc{\proofend}{$\blacksquare$ \vspace{0.5cm}}
\nc{\freerbpl}{{F^{\mathrm RBPL}}}
\nc{\sha}{{\mbox{\cyr X}}}  
\nc{\ncsha}{{\mbox{\cyr X}^{\mathrm NC}}} \nc{\ncshao}{{\mbox{\cyr
X}^{\mathrm NC,\,0}}}
\nc{\shpr}{\diamond}    
\nc{\shprm}{\overline{\diamond}}    
\nc{\shpro}{\diamond^0}    
\nc{\shprr}{\diamond^r}     
\nc{\shpra}{\overline{\diamond}^r}
\nc{\shpru}{\check{\diamond}} \nc{\catpr}{\diamond_l}
\nc{\rcatpr}{\diamond_r} \nc{\lapr}{\diamond_a}
\nc{\sqcupm}{\ot}
\nc{\lepr}{\diamond_e} \nc{\vep}{\varepsilon} \nc{\labs}{\mid\!}
\nc{\rabs}{\!\mid} \nc{\hsha}{\widehat{\sha}}
\nc{\lsha}{\stackrel{\leftarrow}{\sha}}
\nc{\rsha}{\stackrel{\rightarrow}{\sha}} \nc{\lc}{\lfloor}
\nc{\rc}{\rfloor}
\nc{\tpr}{\sqcup}
\nc{\nctpr}{\vee}
\nc{\plpr}{\star}
\nc{\rbplpr}{\bar{\plpr}}
\nc{\sqmon}[1]{\langle #1\rangle}
\nc{\forest}{\calf}
\nc{\altx}{\Lambda_X} \nc{\vecT}{\vec{T}} \nc{\onetree}{\bullet}
\nc{\Ao}{\check{A}}
\nc{\seta}{\underline{\Ao}}
\nc{\deltaa}{\overline{\delta}}
\nc{\trho}{\widetilde{\rho}}

\nc{\rpr}{\circ}
\nc{\dpr}{{\tiny\diamond}}
\nc{\rprpm}{{\rpr}}

\nc{\mmbox}[1]{\mbox{\ #1\ }} \nc{\ann}{\mrm{ann}}
\nc{\Aut}{\mrm{Aut}} \nc{\can}{\mrm{can}}
\nc{\twoalg}{{two-sided algebra}\xspace}
\nc{\colim}{\mrm{colim}}
\nc{\Cont}{\mrm{Cont}} \nc{\rchar}{\mrm{char}}
\nc{\cok}{\mrm{coker}} \nc{\dtf}{{R-{\rm tf}}} \nc{\dtor}{{R-{\rm
tor}}}
\renewcommand{\det}{\mrm{det}}
\nc{\depth}{{\mrm d}}
\nc{\Div}{{\mrm Div}} \nc{\End}{\mrm{End}} \nc{\Ext}{\mrm{Ext}}
\nc{\Fil}{\mrm{Fil}} \nc{\Frob}{\mrm{Frob}} \nc{\Gal}{\mrm{Gal}}
\nc{\GL}{\mrm{GL}} \nc{\Hom}{\mrm{Hom}} \nc{\hsr}{\mrm{H}}
\nc{\hpol}{\mrm{HP}} \nc{\id}{\mrm{id}} \nc{\im}{\mrm{im}}
\nc{\incl}{\mrm{incl}} \nc{\length}{\mrm{length}}
\nc{\LR}{\mrm{LR}} \nc{\mchar}{\rm char} \nc{\NC}{\mrm{NC}}
\nc{\mpart}{\mrm{part}} \nc{\pl}{\mrm{PL}}
\nc{\ql}{{\QQ_\ell}} \nc{\qp}{{\QQ_p}}
\nc{\rank}{\mrm{rank}} \nc{\rba}{\rm{RBA }} \nc{\rbas}{\rm{RBAs }}
\nc{\rbpl}{\mrm{RBPL}}
\nc{\rbw}{\rm{RBW }} \nc{\rbws}{\rm{RBWs }} \nc{\rcot}{\mrm{cot}}
\nc{\rest}{\rm{controlled}\xspace}
\nc{\rdef}{\mrm{def}} \nc{\rdiv}{{\rm div}} \nc{\rtf}{{\rm tf}}
\nc{\rtor}{{\rm tor}} \nc{\res}{\mrm{res}} \nc{\SL}{\mrm{SL}}
\nc{\Spec}{\mrm{Spec}} \nc{\tor}{\mrm{tor}} \nc{\Tr}{\mrm{Tr}}
\nc{\mtr}{\mrm{sk}}

\nc{\ab}{\mathbf{Ab}} \nc{\Alg}{\mathbf{Alg}}
\nc{\Algo}{\mathbf{Alg}^0} \nc{\Bax}{\mathbf{Bax}}
\nc{\Baxo}{\mathbf{Bax}^0} \nc{\RB}{\mathbf{RB}}
\nc{\DA}{\mathbf{DA}}
\nc{\RBo}{\mathbf{RB}^0} \nc{\BRB}{\mathbf{RB}}
\nc{\Dend}{\mathbf{DD}} \nc{\bfk}{{\bf k}} \nc{\bfone}{{\bf 1}}
\nc{\base}[1]{{a_{#1}}} \nc{\detail}{\marginpar{\bf More detail}
    \noindent{\bf Need more detail!}
    \svp}
\nc{\Diff}{\mathbf{Diff}} \nc{\gap}{\marginpar{\bf
Incomplete}\noindent{\bf Incomplete!!}
    \svp}
\nc{\FMod}{\mathbf{FMod}} \nc{\mset}{\mathbf{MSet}}
\nc{\rb}{\mathrm{RB}} \nc{\Int}{\mathbf{Int}}
\nc{\da}{\mathrm{DA}}
\nc{\Mon}{\mathbf{Mon}}
\nc{\remarks}{\noindent{\bf Remarks: }}
\nc{\OS}{\mathbf{OS}} 
\nc{\Rep}{\mathbf{Rep}}
\nc{\Rings}{\mathbf{Rings}} \nc{\Sets}{\mathbf{Sets}}
\nc{\DT}{\mathbf{DT}}

\nc{\BA}{{\mathbb A}} \nc{\CC}{{\mathbb C}} \nc{\DD}{{\mathbb D}}
\nc{\EE}{{\mathbb E}} \nc{\FF}{{\mathbb F}} \nc{\GG}{{\mathbb G}}
\nc{\HH}{{\mathbb H}} \nc{\LL}{{\mathbb L}} \nc{\NN}{{\mathbb N}}
\nc{\QQ}{{\mathbb Q}} \nc{\RR}{{\mathbb R}} \nc{\BS}{{\mathbb{S}}} \nc{\TT}{{\mathbb T}}
\nc{\VV}{{\mathbb V}} \nc{\ZZ}{{\mathbb Z}}


\nc{\calao}{{\mathcal A}} \nc{\cala}{{\mathcal A}}
\nc{\calb}{\mathcal B}
\nc{\calc}{{\mathcal C}} \nc{\cald}{{\mathcal D}}
\nc{\cale}{{\mathcal E}} \nc{\calf}{{\mathcal F}}
\nc{\calfr}{{{\mathcal F}^{\,r}}} \nc{\calfo}{{\mathcal F}^0}
\nc{\calfro}{{\mathcal F}^{\,r,0}} \nc{\oF}{\overline{F}}
\nc{\calg}{{\mathcal G}} \nc{\calh}{{\mathcal H}}
\nc{\cali}{{\mathcal I}} \nc{\calj}{{\mathcal J}}
\nc{\call}{{\mathcal L}} \nc{\calm}{{\mathcal M}}
\nc{\caln}{{\mathcal N}} \nc{\calo}{{\mathcal O}}
\nc{\calp}{{\mathcal P}} \nc{\calq}{{\mathcal Q}} \nc{\calr}{{\mathcal R}}
\nc{\calt}{{\mathscr T}} \nc{\caltr}{{\mathcal T}^{\,r}}
\nc{\calu}{{\mathcal U}} \nc{\calv}{{\mathcal V}}
\nc{\calw}{{\mathcal W}} \nc{\calx}{{\mathcal X}}
\nc{\CA}{\mathcal{A}}
\nc{\cals}{{\mathcal S}}

\nc{\fraka}{{\mathfrak a}} \nc{\frakB}{{\mathfrak B}}
\nc{\frakb}{{\mathfrak b}} \nc{\frakd}{{\mathfrak d}}
\nc{\oD}{\overline{D}}
\nc{\frakF}{{\mathfrak F}} \nc{\frakg}{{\mathfrak g}}
\nc{\frakm}{{\mathfrak m}} \nc{\frakM}{{\mathfrak M}}
\nc{\frakMo}{{\mathfrak M}^0} \nc{\frakp}{{\mathfrak p}}
\nc{\frakS}{{\mathfrak S}} \nc{\frakSo}{{\mathfrak S}^0}
\nc{\fraks}{{\mathfrak s}} \nc{\os}{\overline{\fraks}}
\nc{\frakT}{{\mathfrak T}}
\nc{\oT}{\overline{T}}
\nc{\frakX}{{\mathfrak X}} \nc{\frakXo}{{\mathfrak X}^0}
\nc{\frakx}{{\mathbf x}}
\nc{\frakTx}{\frakT}      
\nc{\frakTa}{\frakT^a}        
\nc{\frakTxo}{\frakTx^0}   
\nc{\caltao}{\calt^{a,0}}   
\nc{\ox}{\overline{\frakx}} \nc{\fraky}{{\mathfrak y}}
\nc{\frakz}{{\mathfrak z}} \nc{\oX}{\overline{X}}
\nc{\frakP}{{\mathfrak P}}

\font\cyr=wncyr10

\nc{\redtext}[1]{\textcolor{red}{#1}}
\nc{\cm}[1]{\textcolor{purple}{Chengming: #1}}
\nc{\li}[1]{\textcolor{red}{Li: #1}}
\nc{\jun}[1]{\textcolor{blue}{Jun: #1}}


\title{Splitting of operads and Rota-Baxter operators on operads}
\author{Jun Pei}
\address{Department of Mathematics, Lanzhou University, Lanzhou, Gansu 730000, China}
         \email{peitsun@163.com}

\author{Chengming Bai}
\address{Chern Institute of Mathematics \& LPMC, Nankai University, Tianjin 300071, China}
         \email{baicm@nankai.edu.cn}

\author{Li Guo}
\address{Department of Mathematics and Computer Science,
         Rutgers University,
         Newark, NJ 07102}
\email{liguo@rutgers.edu}

\date{\today}


\begin{abstract}
This paper establishes a uniform procedure to split the operations in any algebraic operad, generalizing previous known notions of splitting algebraic structures from the dendriform algebra of Loday that splits the associative operation to the successors that split any binary operad. Examples are provided for various $n$-associative algebras, $n$-Lie algebras, $A_\infty$ algebras and $L_\infty$ algebras. Further, the concept of a Rota-Baxter operator, first showing its importance in the associative and Lie algebra context and then generalized to any binary operads, is generalized to arbitrary operads. The classical links from the Rota-Baxter associative algebra to the dendriform algebra and its numerous generalizations are further generalized and unified as the link from the Rota-Baxter operator on an operad to the splitting of the operad. Finally, the remarkable fact that any dendriform algebra can be recovered from a relative Rota-Baxter operator is generalized to the context of operads with the generalized notion of a relative Rota-Baxter operator for any operad.
\end{abstract}

\subjclass[2010]{18D50, 17A30, 17A36, 17B99, 17C99}

\keywords{Operads, $n$-algebra, Rota-Baxter operator, dendriform algebra, pre-Lie algebra, $A_\infty$ algebra, $L_\infty$ algebra, splitting, successor}

\maketitle

\tableofcontents

\setcounter{section}{0}

\section{Introduction}


Dendriform (di)algebra~\mcite{Lo2} is a module with two binary operations whose sum is associative, thus giving a two-part splitting of the associativity. This concept was introduced by Loday in the late 1990s with motivation from periodicity in algebraic $K$-theory. Several years later, Loday and Ronco~\mcite{LR} introduced the concept of a tridendriform algebra (previous called dendriform trialgebra) from their study of algebraic topology. It is a module with three binary operations whose sum is associative, thus giving a three-part splitting of the associativity. Subsequently, quite a few similar algebraic structures were introduced, such as the quadi-algebra~\mcite{AL} and ennea algebra~\mcite{Le3}. The notion of splitting of associativity was introduced by Loday~\mcite{Lo5} to describe this phenomena in general for the associative operation (see also~\mcite{EG2,Hol1}).

The splitting of the associativity turns out be important both in theory and application. The free objects for dendriform and tridendriform algebras equipped the planar binary trees and planar trees with a natural Hopf algebraic structure which is closely related to the Connes-Kreimer Hopf algebra of rooted trees from their study of quantum field theory.  A similar two-part and three-part splittings of the Lie operation are found to be respectively the pre-Lie algebra predating dendriform algebra with broad connections~\mcite{Kos,Vi,Ger,Bur}
and the PostLie algebra from operadic study~\mcite{Va} with applications to integrable systems~\mcite{BGN1}. Further, a two-part and three-part splittings of the associative commutative operation give the Zinbiel algebra~\mcite{Lo2} and commutative tridendriform algebra~\mcite{Zi} respectively. The free objects in the categories of these two algebras are respectively the shuffle algebras and quasi-shuffle algebras, thus providing algebraic characterization of these two important algebras say in the study of multiple zeta values~\mcite{Gub,Ho1,Ho2,Lo4}.

Analogues of the dendriform algebra and tridendriform algebra for the Jordan algebra, alternative algebra and Poisson algebra have also been obtained~\mcite{Ag2,BLN,HNB,LNB,NB}.
To put all these constructions in one framework, the concepts of a disuccesor and a trisuccessor were introduced in~\mcite{BBGN}, giving two-part and three-part splittings of any binary operad, relating them to the Manin black product~\mcite{Va} and the Rota-Baxter operator on a binary operad~\mcite{Ba,Gub,R1}.

\smallskip

There are important algebraic structures (operads) beyond the binary ones, such as the various $n$-associative and Lie algebras, the $A_\infty$ algebra~\mcite{St} and $L_\infty$ algebra~\mcite{BM,BM1,GGR,LV,MSS}. It can be expected that splittings of these operads will also show their importance as in the binary case. In fact, the $Dend_\infty$ algebra~\mcite{LV,Zi} and the $PL_\infty$ algebra~\mcite{Ch} have been defined that should be suitable splittings of the $A_\infty$ algebra and the $L_\infty$ algebra. Structures have also arisen recently that resemble a two-part splitting for the 3-Lie algebra~\mcite{BGS}. Instead of discovering such structures one at a time in an ad hoc manner and with elaborating experiments, it is desirable from the theoretical and application perspective to establish a general framework for the splittings of all operads, generalizing the approach for binary operads in~\mcite{BBGN}. This is the purpose of this paper.
In fact, we generalize~\mcite{BBGN} in two directions. In one direction, we generalize the arity of the operads under consideration from binary to any (uniform or mixed arities). In the other direction, for a given arity or arity combination, we introduce the concept of a \conf to give a uniform treatment of different splitting patterns that include the bisuccessor and trisuccessor in~\mcite{BBGN} as two special cases. This avoids repetitive arguments and paves the way for understanding the other splittings of the associativity beyond the dendriform and tridendriform algebras.


The Rota-Baxter operator which has played important role in broad areas in mathematics and physics~\mcite{Ba,Ca,CK,Gub,R1} naturally gives rise to splittings of various algebraic structures through its action on these structures, providing interesting examples and motivation for the splittings. This is the case for associative algebras, giving the dendriform and tridendriform algebras, for Lie algebras, giving the pre-Lie and PostLie algebras, and more generally for binary operads, giving bisuccessors and trisuccessors~\mcite{BBGN}. Going in the opposite direction, any dendriform and tridendriform algebras can be recovered in this fashion by a generalization of Rota-Baxter operators, called relative Rota-Baxter operators~\mcite{BGN,Uc}. We generalize these results to algebras of any operads. In order to do this, we generalize the concepts of a Rota-Baxter operator and a relative Rota-Baxter operator to the context of an operad.
\smallskip


The following is an outline of the paper. In Section~\mref{sec:conc}, we define the splitting of labeled trees from a given splitting pattern called a \conf. Through the tree description of operads, splittings of operads are defined. In Section~\mref{sec:exam}, examples of splittings of operads are provided for various associative $n$-algebras and $n$-Lie algebras, the $A_\infty$ algebra and the $L_\infty$ algebra. In Section~\mref{sec:spl}, we show that a splitting of an operad indeed satisfies the same splitting property for the operations of the given operad as in the previous known cases of splittings (successors)~\mcite{BBGN}, thus justifying the name of the concept of splitting. Functorial properties of the splitting process are also studied. In Section~\mref{sec:rb}, the concept of a Rota-Baxter operator on an operad is defined with respect to a \conf. It is shown that a Rota-Baxter operator action on an operad induces a splitting of the operad. To address the question of whether any splitting of an operad can be derived from some action of Rota-Baxter type, the concept of a relative Rota-Baxter operator (also called an $\calo$-operator or a generalized Rota-Baxter operator in special cases considered before) was introduced. It is shown that, as in the case of associative algebras, any algebra of a splitting of an operad comes from the action of a relative Rota-Baxter operator on an algebra of the original operad.

\section{Splittings of operads}
\mlabel{sec:conc}
In this section, we extend splittings of binary operads introduced in \cite{BBGN} to arbitrary operads. To handle the different splitting patterns of a given operad, we define in Section~\mref{ss:tree} the notion of a \conf for a splitting. Since operads can be represented by trees, for each \conf, we define a splitting of labeled trees which is then applied to define similar splitting for nonsymmetric operads and (symmetric) operads.

\subsection{Splittings of planar trees}
\mlabel{ss:tree}
We recall some basic notions on trees and operads. For more details see~\mcite{BBGN,LV}.
\subsubsection{Labeled trees}
\begin{defn}
{\rm
\begin{enumerate}
\item
Let $\calt$ denote the set of planar reduced rooted trees together with the trivial tree $\vcenter{\xymatrix@M=4pt@R=8pt@C=4pt{\ar@{-}[d]\\ \\}}$. If $t\in \calt$ has $n$ leaves, we call $t$ an {\bf $n$-tree}. The trivial tree $\vcenter{\xymatrix@M=4pt@R=8pt@C=4pt{\ar@{-}[d]\\ \\}}$ has one leaf.
\item
 Let $\vertset$ be a set. By a {\bf decorated tree} we mean a tree $t$ of $\calt$ together with a decoration on the vertices of $t$ by elements of $\vertset$ and a decoration on the leaves of $t$ by distinct positive integers. Let $t(\vertset)$ denote the set of decorated trees of $t$ and denote
\begin{equation*}
\calt(\vertset):=\coprod_{ t \in \calt} t(\vertset).
\end{equation*}
If $\tau\in t(\vertset)$ for an $n$-tree $t$, we call $\tau$ a {\bf labeled $n$-tree}.
\item
For $\tau\in \calt(\vertset)$, we let $\vin(\tau)$ (resp. $\lin(\tau)$) denote the set (resp. ordered set) of labels of the vertices (resp. leaves) of $\tau$.
\item
Let $\tau\in \calt(\vertset)$ with $|\lin(\tau)|>1$ be a labeled tree from $t\in \calt$. Then there exists an integer $m$ such that $t$ can be written uniquely as the grafting $t_{1} \vee t_{2} \vee \cdots \vee t_{m}$ of trees $t_{1},t_{2}, \cdots t_{m}$. Correspondingly, let $\tau= \omega (\tau_{1} \vee \tau_{2} \vee \cdots \vee \tau_{m})$ denote the unique decomposition of $\tau$ as a grafting of $\tau_{1},\cdots,\tau_{m}$ in $\calt(\vertset)$ along $\gop\in \vertset$.
\end{enumerate}}
\end{defn}

For a tree $t$ and an arity graded vector space $V=\bigoplus\limits_{n\geq 1}V_n$, we define
$$
t[\gensp] := \displaystyle{\bigotimes_{v \in \vin(t)}} \gensp_{|\inv(v)|},
$$
where $\inv(v)$ denotes the incoming edges of $v$, labeled by $1,\cdots,|\inv(v)|$ from the left to the right. Then the free nonsymmetric operad $\mathcal{T}_{\hspace*{-0.1cm}ns}(\gensp)$ on $V$ is given by the vector space
$$
\mathcal{T}_{\hspace*{-0.1cm}ns}(\gensp) := \displaystyle{ \bigoplus_{t \in \calt } } \ t[\gensp].
$$
A basis $\genbas$ of $\gensp$ induces a basis $t(\genbas)$ of $t[\gensp]$ and a basis $\calt(\genbas)$ of $\mathcal{T}_{\hspace*{-0.1cm}ns}(\gensp)$. Consequently any element of $t[\gensp]$ can be represented as a linear combination of elements in $t(\genbas)$.

\subsubsection{Configurations and splittings of labeled trees}
For each integer $n\geq 1$, denote $[n]= \{1,2,\cdots,n\}$.
\begin{defn}
{\rm
\begin{enumerate}
\item
For any $1\leq m\leq n$, let $N_{(n,m)}$ denote the set of all nonempty subsets of $[n]$ with at most $m$ elements. In particular,
$$N_{(n,1)}:=A_{n}: = \{\{1\},\{2\},\cdots, \{n\}\}, \quad
N_{(n,n)}:=B_n:=\{J\subseteq [n]\,|\,J\neq \emptyset\}.$$
\item Let $\tau\in \calt(\vertset)$ and $\emptyset \neq J\subseteq \lin(\tau)$. For $\gop\in \vin(\tau)$, let $\tau_\gop$ denote the subtree of $\tau$ with root $\gop$ and let $\tau_\gop=\gop(\sigma_1\vee \cdots \vee \sigma_\ell)$ be the decomposition of $\tau_\gop$ as the grafting of decorated branches of $\tau_\gop$. Denote
\begin{equation}
J\sqcap \gop: = J\sqcap (\gop;\tau):= \{ i\in \inv(\gop)\,|\,J\cap \lin(\tau_i)\neq \emptyset\} \subseteq [\ell].
\mlabel{eq:closed}
\end{equation}
\end{enumerate}
}
\mlabel{de:conf}
\end{defn}

To get use the notation $J\sqcap \gop$, consider the case when $\lin(\tau)=[3]$. There are three such $\tau$'s. Then $\omega\in \vin(\tau)$ can have arity 2 or 3. If $|\vin(\omega)|=3$, then $\tau_\gop=\tau$ and $J\sqcap \gop=J$ for each $\emptyset\neq J\subseteq [3]$. If $|\vin(\omega)|=2$, then $\omega$ can appear in $\tau$ in four locations denoted $\omega_i, 1\leq i\leq 4$: $$\tau=\omega_1(1\vee \omega_2(2\vee 3)), \quad \tau=\omega_3(\omega_4(1\vee 2)\vee 3).$$
For $J={2}$, we have
$$ J\sqcap \omega_1 = \{2\}, J\sqcap \omega_2=\{1\}, J\sqcap \omega_3=\{1\}, J\sqcap \omega_4 =\{2\}.$$
For $J=\{1,3\}$, we have
$$ J\sqcap \omega_1 = \{1,2\}, J\sqcap \omega_2=\{2\}, J\sqcap \omega_3=\{1,2\}, J\sqcap \omega_4 =\{1\}.$$

\begin{defn}
{\rm
\begin{enumerate}
\item
A {\bf \conf} is a sequence $\confm=(\confn_n)_{n\geq 1}$ with $\confn_n \subseteq B_{n}$  such that for any $J \in \confn_n$, decorated $n$-tree $\tau$ and $\gop\in \vin(\tau)$, we have $J\sqcap \gop\in C_{|\inv(\gop)|}$ whenever $J\sqcap\gop\neq \emptyset$.
\item
For a \conf $\confm=(\confn_n)$, define $1\leq p(\confm)\leq \infty$ by
\begin{equation}
    p(\confm):=\sup \{n\,|\, \confn_n=B_n\},
\mlabel{eq:ind}
\end{equation}
if it exists, called the {\bf index} of $\confm$.
\item
A \conf $\confm=(\confn_n)$ is called {\bf $\BS$-invariant} if $\confn_n^{\BS_n}\subseteq \confn_n, n\geq 1$.
\end{enumerate}
}
\mlabel{de:cconf}
\end{defn}

From the computations before Definition~\mref{de:conf}, we find that $\{2\}\in \confn_3$ implies $\{1\}, \{2\}\in \confn_2$ and $\{1,3\}\in \confn_3$ implies $\{1\},\{2\},\{1,2\}\in \confn_2$.

It is easy to see that if the $\BS$-invariant $\confn_{n}$ contains a subset of $p$ elements, then $\confn_{n}$ contains all the subsets of $[n]$ with $p$ elements.

\begin{exam}
{\rm
\begin{enumerate}
\item Any \conf has index $n\geq 1$. An important example of \conf with index 1 is $\cala:=(A_n)$.
\item The sequence $\confm=(\confn_n)$ with $$
\confn_{n} = \left\{\begin{array}{ll} B_{n},  & 1 \leq n \leq m ,   \\
N_{(n,m)}, & n > m .\end{array}\right.
$$
is an $\BS$-invariant \conf with index $m$. The \conf $\calb:=(B_n)$ has index $\infty$.
\item The sequence with $\confn_{n} = \{[n]\}$ is an $\BS$-invariant \conf with no index. It is called the {\bf trivial ($\BS$-invariant) \conf}.
\item For any fixed $m\geq 1$, the sequence with $\confn_{n} = \{[n]\}$, $n \geq m$, $\confn_{n} =\{[n],\{1\},\{2\}, \cdots, \{n\}\}$, $n <m$, is an $\BS$-invariant \conf.
\end{enumerate}
}
\mlabel{ex:conf}
\end{exam}

\begin{defn} \mlabel{defn:vector}
{\rm
Let $\gensp=\bigoplus\limits_{n\geq 1}V_n$ be an arity graded vector space with basis $\genbas=\coprod_{n\geq 1}\genbas_{n}$ and $\confm$ be a \conf.
\begin{enumerate}
\item
Define an arity graded vector space $\ssp(V)$ by
\begin{equation}
\ssp(V)_{n}=V_{n} \ot \left(\bigoplus_{I \in \confn_{n}}   \bfk e_{I}\right),
\mlabel{eq:dsp}
\end{equation}
where we denote $(\gop \otimes e_{I})$ by $\svec{\gop}{e_{I}}$ for $\gop\in \genbas_{n}$. Then $\left\{\svec{\gop}{e_{I}}\,\Big|\, \gop \in \genbas_{n}, n\geq 1, I \in \confn_{n} \right\}$ is a basis of $\ssp(V)_{n}$.
\item
For a labeled $n$-tree $\tau$ in $\calt(\genbas)$, define $\ssp(\tau)$, also denoted by $\ssp_\emptyset(\tau)$, in $\mathcal{T}_{\hspace*{-0.1cm}ns}(\ssp(V))$ by
\begin{itemize}
 \item[$\bullet$] $\ssp(\,\vcenter{\xymatrix@M=0pt@R=8pt@C=4pt{\ar@{-}[d]\\ \\}}\,)= \vcenter{\xymatrix@M=4pt@R=8pt@C=4pt{\ar@{-}[d]\\ \\}}$,
\item[$\bullet$] when $n\geq 2$, $\ssp(\tau)$ is obtained by replacing each decoration $\gop \in \vin(\tau) \cap \genbas_{\ell}$ by $\svec{\gop}{\ast}:=\svec{\gop}{\ast_\confm}= \sum_{I \in \confn_{\ell}} \svec{\gop}{e_{I}}.$
\end{itemize}
We extend this definition to $\mathcal{T}_{\hspace*{-0.1cm}ns}(\gensp)$ by linearity.
\end{enumerate}
}
\end{defn}

\begin{defn}
{\rm
Let $\gensp=\bigoplus\limits_{n\geq 1}V_n$ be an arity graded vector space with basis $\genbas=\coprod_{n\geq 1}\genbas_{n}$ and let $\confm$ be a \conf. Let $\tau\in \calt(\genbas)$ be a labeled $n$-tree and let $J \in \confn_{|\lin(\tau)|}$. The {\bf splitting with \conf $\confm$} (or {\bf \spl-splitting} in short) $\ssp_{J}(\tau) $ of $\tau$ with respect to $J$ is an element of $ \mathcal{T}_{\hspace*{-0.1cm}ns}(\ssp(\gensp))$ defined by induction on $n:=|\lin(\tau)|$ as follows:
\begin{enumerate}
\item[$\bullet$] $\ssp_J(\vcenter{\xymatrix@M=4pt@R=8pt@C=4pt{\ar@{-}[d]\\ \\}})=\vcenter{\xymatrix@M=0pt@R=8pt@C=4pt{\ar@{-}[d]\\ \\}}$ \ ;
\item[$\bullet$] assume that $\ssp_J(\tau)$ have been defined for $\tau$ with $|\lin(\tau)|\leq k$ for a $k\geq 1$. Then, for a labeled $(k+1)$-tree $\tau\in \calt(\genbas)$ with its decomposition $\tau=\graft$ and $\omega \in \genbas_{\ell}$, denote $I: =J\sqcap \tau \in \confn_{\ell} $ defined in Eq.~(\mref{eq:closed}) and define
$$\ssp_{J}(\tau):=\ssvec{\gop}{e_{I}}(\vee_{i=1}^\ell ~\ssp_{J\cap\lin(\tau_i)}(\tau_{i}) )=\ssvec{\gop}{e_{I}} (~\ssp_{J\cap\lin(\tau_1)}(\tau_{1}) \vee \cdots  \vee \ssp_{J\cap\lin(\tau_\ell)}(\tau_{\ell})~),
$$
using the notation $\ssp_{\emptyset}(\tau)=\ssp(\tau)$ from the previous definition.
\end{enumerate}
}\mlabel{de:vtilo}
\end{defn}

We single out two important splittings that specialize to the two classical examples of splittings, namely the bisuccessor and the trisuccessor~\mcite{BBGN} in the case of binary operads.

\begin{defn}
{\rm
With the notation in Definition~\mref{de:conf}, the $\cala$-splitting (resp. $\calb$-splitting) is called the {\bf arity-splitting} (resp. {\bf power-splitting}).
}
\end{defn}

\begin{remark}
{\rm
For a binary tree $\tau$, the arity-splitting $\dsu_J(\tau)$ and power-splitting $\tsu_J(\tau)$ are the bisuccessor and trisuccessor of $\tau$ in~\cite{BBGN} respectively.
For the trivial \conf $\confm=([n])$, we have $\ssp_J(\tau)=\ssvec{\tau}{e_{[|\lin(\tau)|]}}$ resulting in no splittings, justifying the term ``trivial".
}
\mlabel{rm:bin}
\end{remark}

We next give an explicit description of the \spl-splitting.  \begin{prop}\label{pp:psupath}
Let $\gensp=\bigoplus\limits_{n\geq 1}V_n$ be an arity graded vector space with basis $\genbas=\coprod_{n\geq 1}\genbas_{n}$ and $\confm$ be a \conf. Let $\tau$ be in $\calt(\genbas)$ and $J \in \confn_{|\lin(\tau)|}$.  With the notations in Definition~\mref{de:cconf}, the \spl-splitting $\ssp_{J}(\tau)$ is obtained by relabeling each vertex $\omega\in \genbas_{\ell}$ of $\tau$ by $\svec{\gop}{e_I}$ if $I:=J\sqcap\gop\neq \emptyset$ and by $\svec{\gop}{\ast_\confm}:= \displaystyle \sum_{I \in \confn_{\ell}} \svec{\gop}{e_{I}}$ if $J\sqcap\gop= \emptyset$.
\end{prop}
\begin{proof}
The proof follows from an induction on $|\lin(\tau)|$.
\mlabel{pp:exp}
\end{proof}

\begin{exam}
{\rm
\begin{enumerate}
\item
For the \conf $\cala$, we have
$$\ssp_{\{x_{2}\}} \left( \begin{xy}
(0,8)*+{x_{1}}="v1";(10,8)*+{x_{2}}="v2";(20,8)*+{x_{3}}="v3"; (30,8)*+{x_{4}}="v4"; (38,8)*+{x_{5}}="v5"; (46,8)*+{x_{6}}="v6";
(20,0)*+{\omega_{1}}="v7";(42,0)*+{\omega_{2}}="v8"; (20,-8)*+{\omega_{3}}="v9";(20,-16)*+{}="v10";
{\ar@{-}"v10";"v9"};{\ar@{-}"v9";"v8"};{\ar@{-}"v9";"v7"};{\ar@{-}"v9";"v1"};{\ar@{-}"v9";"v8"};{\ar@{-}"v7";"v2"};{\ar@{-}"v7";"v3"};
{\ar@{-}"v7";"v4"};{\ar@{-}"v8";"v5"};{\ar@{-}"v8";"v6"}
\end{xy} \right)= \begin{xy}
(0,9)*+{x_{1}}="v1";(10,9)*+{x_{2}}="v2";(20,9)*+{x_{3}}="v3"; (30,9)*+{x_{4}}="v4"; (38,9)*+{x_{5}}="v5"; (46,9)*+{x_{6}}="v6";
(20,0)*+{\svec{\omega_{1}}{e_{1}}}="v7";(42,-2)* +{\svec{\omega_{2}}{\ast_{\!\!{}_\cala}}}="v8"; (20,-12)*+{\svec{\omega_{3}}{e_{2}}}="v9";(20,-22)*+{}="v10";
{\ar@{->}"v10";"v9"};{\ar@{-}"v9";"v8"};{\ar@{->}"v9";"v7"};{\ar@{-}"v9";"v1"};{\ar@{-}"v9";"v8"};{\ar@{->}"v7";"v2"};{\ar@{-}"v7";"v3"};
{\ar@{-}"v7";"v4"};{\ar@{-}"v8";"v5"};{\ar@{-}"v8";"v6"}
\end{xy}
$$
$$
= \begin{xy}
(0,9)*+{x_{1}}="v1";(10,9)*+{x_{2}}="v2";(20,9)*+{x_{3}}="v3"; (30,9)*+{x_{4}}="v4"; (38,9)*+{x_{5}}="v5"; (46,9)*+{x_{6}}="v6";
(20,0)*+{\svec{\omega_{1}}{e_{1}}}="v7";(42,-2)*+{\svec{\omega_{2}}{e_{1}}}="v8"; (20,-12)*+{\svec{\omega_{3}}{e_{2}}}="v9";(20,-22)*+{}="v10";
{\ar@{-}"v10";"v9"};{\ar@{-}"v9";"v8"};{\ar@{-}"v9";"v7"};{\ar@{-}"v9";"v1"};{\ar@{-}"v9";"v8"};{\ar@{-}"v7";"v2"};{\ar@{-}"v7";"v3"};
{\ar@{-}"v7";"v4"};{\ar@{-}"v8";"v5"};{\ar@{-}"v8";"v6"}
\end{xy} +\begin{xy}
(0,9)*+{x_{1}}="v1";(10,9)*+{x_{2}}="v2";(20,9)*+{x_{3}}="v3"; (30,9)*+{x_{4}}="v4"; (38,9)*+{x_{5}}="v5"; (46,9)*+{x_{6}}="v6";
(20,0)*+{\svec{\omega_{1}}{e_{1}}}="v7";(42,-2)*+{\svec{\omega_{2}}{e_{2}}}="v8"; (20,-12)*+{\svec{\omega_{3}}{e_{2}}}="v9";(20,-22)*+{}="v10";
{\ar@{-}"v10";"v9"};{\ar@{-}"v9";"v8"};{\ar@{-}"v9";"v7"};{\ar@{-}"v9";"v1"};{\ar@{-}"v9";"v8"};{\ar@{-}"v7";"v2"};{\ar@{-}"v7";"v3"};
{\ar@{-}"v7";"v4"};{\ar@{-}"v8";"v5"};{\ar@{-}"v8";"v6"}
\end{xy}
$$
\\
\item For a \conf $\confm=(\confn_n)$ with index $\geq 3$ (for example for $\confm=\calb$), we have\\
$\ssp_{\{x_{1},x_{2},x_{4}\}}  \left( \begin{xy}
(0,8)*+{x_{1}}="v1";(10,8)*+{x_{2}}="v2";(20,8)*+{x_{3}}="v3"; (30,8)*+{x_{4}}="v4"; (40,8)*+{x_{5}}="v5"; (50,8)*+{x_{6}}="v6";
(20,0)*+{\omega_{1}}="v7";(45,0)*+{\omega_{2}}="v8"; (20,-8)*+{\omega_{3}}="v9";(20,-16)*+{}="v10";
{\ar@{-}"v10";"v9"};{\ar@{-}"v9";"v8"};{\ar@{-}"v9";"v7"};{\ar@{-}"v9";"v1"};{\ar@{-}"v9";"v8"};{\ar@{-}"v7";"v2"};{\ar@{-}"v7";"v3"};
{\ar@{-}"v7";"v4"};{\ar@{-}"v8";"v5"};{\ar@{-}"v8";"v6"}
\end{xy} \right)  = \begin{xy}
(0,9)*+{x_{1}}="v1";(10,9)*+{x_{2}}="v2";(20,9)*+{x_{3}}="v3"; (30,9)*+{x_{4}}="v4"; (38,9)*+{x_{5}}="v5"; (46,9)*+{x_{6}}="v6";
(20,0)*+{\svec{\omega_{1}}{e_{\{1,3\}}}} ="v7";(42,-2)*+{\svec{\omega_{2}}{\ast_{\!{}_{\confm}}}}="v8"; (20,-12)*+{\svec{\omega_{3}}{e_{\{1,2\}}}}="v9";(20,-22)*+{}="v10";
{\ar@{->}"v10";"v9"};{\ar@{-}"v9";"v8"};{\ar@{->}"v9";"v7"};{\ar@{->}"v9";"v1"};{\ar@{-}"v9";"v8"};{\ar@{->}"v7";"v2"};{\ar@{-}"v7";"v3"};
{\ar@{->}"v7";"v4"};{\ar@{-}"v8";"v5"};{\ar@{-}"v8";"v6"}
\end{xy}$\\
where $\svec{\omega}{\ast_{\!{}_{\confm}}}=\sum\limits_{I\in \confn_2} e_I.$
\end{enumerate}}
\end{exam}

\subsection{Splittings of nonsymmetric operads and (symmetric) operads}\label{repnsopd}
We now give the splitting of an operad with a given splitting \conf, starting with the nonsymmetric case.
\subsubsection{The nonsymmetric case}
Let $\gensp=\bigoplus\limits_{n\geq 1}V_n$ be an arity graded vector space with basis $\genbas=\coprod_{n\geq 1}\genbas_{n}$.
\begin{enumerate}
\item
An element
$$r:=\sum_{i=1}^r c_{i}\tau_{i}, \quad c_{i}\in\bfk, \tau_i\in \calt(\genbas),$$
in $\mathcal{T}_{\hspace*{-0.1cm}ns}(\gensp)$ is called {\bf homogeneous} if $\lin(\tau_i)$ are the same for $1\leq i\leq r$. Then denote $\lin(r)=\lin(\tau_i)$ for any $1\leq i\leq r$.
\item
A collection of elements
$$r_s:=\sum_{i=1}^r c_{s,i}\tau_{s,i}, \quad c_{s,i}\in\bfk, \tau_{s,i}\in \calt(\genbas), 1\leq s\leq k, k\geq 1, $$
in $\mathcal{T}_{\hspace*{-0.1cm}ns}(\gensp)$ is called {\bf locally homogenous} if each element $r_s$, $1\leq s\leq k$, is homogeneous.
\end{enumerate}

\begin{defn}\mlabel{rule}
{\rm Let $\opd=\mathcal{T}_{\hspace*{-0.1cm}ns}(V)/(R)$ be a nonsymmetric operad where $V$ is an arity graded vector space with a basis $\genbas$ and $R\subseteq \mathcal{T}_{\hspace*{-0.1cm}ns}(\gensp)$ is a subset of locally homogeneous elements:
\begin{equation}
r_s=\sum_i c_{s,i}\tau_{s,i}\ \in \mathcal{T}_{\hspace*{-0.1cm}ns}(\gensp)\; , \; \ c_{s,i}\in\bfk, \ \tau_{s,i}\in \calt(\genbas), \ 1\leq s\leq k.
\end{equation}
Let $\confm=(\confn_n)$ be a \conf. The {\bf \spl-splitting} of $\opd$ is defined to be the operad
$$\ssp(\opd)=\mathcal{T}(\ssp(\gensp))/ (\ssp(\relsp))$$
where the space of relations is generated by
\begin{equation*}
\ssp(R):= \left\{\ssp_{J}(r_{s}) = \sum_{i} c_{s,i} \ssp_{J}(\tau_{s,i})~\big|~  J \in \confn_{|\lin(\tau_{s,i})|}, 1\leq s \leq k  \right\}.
\end{equation*}
}
\end{defn}

\subsubsection{The symmetric case}

Let $\gensp=\bigoplus\limits_{n\geq 1}\gensp(n)$ be an $\BS$-module with a linear basis $\genbas=\coprod_{n\geq 1}\genbas(n)$ such that $\genbas(n)$ is invariant under the action of $\BS_n$.
For any finite set $\mathcal{X}$ of cardinality $n$, define the coinvariant space
$$ \gensp(\mathcal{X}):= \left( \displaystyle{\bigoplus_{f: [n]\rightarrow \mathcal{X}}} \gensp(n) \right)_{\BS_{n}} \, ,$$
where the sum is over all the bijections from $[n]:=\{ 1,2, \ldots , n \}$ to $\mathcal{X}$ and where the symmetric group acts diagonally.

Let $\mathbb{T}$ denote the set of isomorphism classes of reduced trees~\cite[Appendix C]{LV}. For $\textsf{t}\in \mathbb{T}$, define the treewise tensor $\BS$-module associated to $\textsf{t}$, explicitly given by
$$  \textsf{t}[\gensp] := \displaystyle{\bigotimes_{v \in \vin(\textsf{t})} } \gensp(\inv(v)) \ , $$
see~\cite[Section 5.5.1]{LV}.
Then the free operad $\mathcal{T}(\gensp)$ on an $\BS$-module $\gensp$ is given by the $\BS$-module
$$ \mathcal{T}(\gensp) := \displaystyle{ \bigoplus_{\textsf{t} \in \mathbb{T} } } \ \textsf{t}[\gensp] \,. $$

Each tree $\textsf{t}$ in $\mathbb{T}$ can be represented by a planar tree $t$ in $\calt$ by choosing a total order on the set of inputs of each vertex of $\textsf{t}$. Further, $t[\gensp]\cong \mathsf{t}[\gensp]$~\cite[Section 2.8]{Ho}. Fixing such a choice $t$ for each $\textsf{t}\in \mathbb{T}$ gives a subset $\mathfrak{R}\subseteq \calt$ in bijection with $\mathbb{T}$.
Then we have
$$\mathcal{T}(\gensp)\cong \displaystyle{ \bigoplus_{t \in \mathfrak{R} } } \ t[\gensp] \ , $$
allowing us to use the notations in the nonsymmetric case.

\begin{defn}\mlabel{rule1}
{\rm Let $\opd=\mathcal{T}(\gensp)/ (\relsp)$ be an operad where $\gensp=\bigoplus\limits_{n\geq 1}\gensp(n)$ is an $\BS$-module with a linear basis $\genbas=\coprod_{n\geq 1}\genbas(n)$ that is invariant under the action of $\BS_n$ and where the space of relations $(R)$ is generated, as an $\BS$-module, by a set $R$ of locally homogeneous elements
\begin{equation}
r_s:=\sum_i c_{s,i}\tau_{s,i}, \ c_{s,i}\in\bfk, \tau_{s,i} \in \bigcup_{t\in \mathfrak{R}} t(\genbas), \ 1\leq s\leq k.
\mlabel{eq:pres}
\end{equation}
Let $\confm$ be an $\BS$-invariant \conf. The {\bf \spl-splitting} of $\opd$ is defined to be the operad
$$\ssp(\opd)=\mathcal{T}(\ssp(\gensp))/ (\ssp(\relsp))$$
where the $\BS_n$-action on $\ssp(\gensp)(n)=\gensp(n) \ot (\bigoplus_{I \in \confn_{n}}   \bfk e_{I})$  is given by
\begin{equation*}
\svec{\gop}{e_{I}}^{\sigma}:=\svec{\gop^{\sigma}}{e_{\sigma(I)}}, \quad  \gop\in \gensp(n), \quad \sigma(I) = \{\sigma(i)~|~ i \in I\}
\end{equation*}
and the space of relations $(\ssp(\relsp))$ is generated, as an $\BS$-module, by
\begin{equation*}
\ssp(R):= \left\{\ssp_{J}(r_{s}) = \sum_{i} c_{s,i} \ssp_{J}(\tau_{s,i})~\big|~ J \in \confn_{|\lin(\tau_{s,i})|}, 1\leq s \leq k  \right\}.
\end{equation*}
}
\end{defn}

By Remark~\mref{rm:bin}, we have
\begin{prop}
When $\calp$ is a (symmetric or nonsymmetric) binary operad, the arity and power splitting of $\calp$ is the disuccessor and the trisuccessor of $\calp$ respectively.
\end{prop}

\section{Examples of splittings of operads}
\mlabel{sec:exam}
We now give some examples of splittings of operads, first in the nonsymmetric case in Section~\mref{ss:egns} and then in the general case in Section~\mref{ss:egs}. We will focus on the arity- and power-splittings. But note that there are other splittings, for example from the \confs in Example~\mref{ex:conf}. See~\S\mref{sss:dend}.

\subsection{Examples of splittings of nonsymmetric operads}
\mlabel{ss:egns}
We start with the dendriform algebras which is the origin of all the splitting constructions. We then consider the $n$-ary generalizations. We finally show that the operad $\Dend_\infty$ defined in~\mcite{LV,Zi} is the arity splitting of the operad $A_\infty$~\mcite{St}.

\subsubsection{Dendriform operads revisited}
\mlabel{sss:dend}

Recall that the tridendriform algebra of  Loday and Ronco \mcite{LR} is defined by three bilinear operations $\{\prec, \succ, \cdot\}$ satisfying the following relations:
\begin{eqnarray}
&(x \prec y) \prec z = x \prec (y \ast z), \quad (x \succ y) \prec z = x \succ (y \prec z), \quad (x \ast y) \succ z = x \succ (y \succ z),&  \label{eq:d1}\\
&(x \cdot y) \prec z = x \cdot (y \prec z), \quad (x \prec y) \cdot z = x \cdot (y \succ z), \quad (x \succ y) \cdot z = x \succ (y \cdot z),& \label{eq:d2}\\
&(x \cdot y) \cdot z = x \cdot (y \cdot z).\label{eq:d3}&
\end{eqnarray}
where $\ast = \prec + \succ +\cdot$. The dendriform algebra of Loday \mcite{Lo2} is defined by two bilinear operations $\{\prec, \succ\}$ satisfying the relations in Eq. (\ref{eq:d1}), where $ \ast = \prec + \succ$. It is easy to check that the corresponding nonsymmetric operad Dend (resp. TriDend) is the arity-splitting (resp. power-splitting) of the nonsymmetric operad $As$ of associative algebras.
Let $\confm = (\confn_{n})$ be a \conf with index $2$. Then a $\ssp(\mathit{As})$-algebra, is a vector space $A$ with three bilinear operations $\{\prec,\succ,\cdot\}$ satisfying the relations in Eqs. (\ref{eq:d1}) and (\ref{eq:d2}), where $\ast = \prec + \succ + \cdot$, thus gives a splitting of $As$ between $Dend$ and $TriDend$.

\subsubsection{Dendriform $n$-operads}
There is no unique $n$-arity generalization of the associative algebra for $n\geq 3$. Recall that a {\bf partially associative $n$-algebra} \cite{GGR} is a vector space with an $n$-ary operation such that the signed sum of the ordered product of $2n-1$ elements is zero, that is,
\begin{equation}
\sum_{i=0}^{n-1} (-1)^{i (n-1)} (x_{1}, \cdots ,x_{i}, (x_{i+1},\cdots ,x_{i+n}),x_{i+n+1}, \cdots , x_{2n-1}) = 0.
\mlabel{eq:pas}
\end{equation}
When $n=2$, this reduces to the classical associativity $(x_{1}x_{2})x_{3} - x_{1}(x_{2}x_{3}) =0$.
For the case when $n=3$, the partial associativity is
$$
((x_{1},x_{2},x_{3}),x_{4},x_{5}) + (x_{1},(x_{2},x_{3},x_{4}),x_{5}) + (x_{1},x_{2},(x_{3},x_{4},x_{5}))=0.
$$

A vector space with an $n$-ary operation is called an {\bf (totally) associative $n$-algebra} \cite{BM,GGR} if the ordered product of $2n-1$ elements does not depend on the position of the parentheses, that is,
$$
(x_{1} \cdots (x_{i},\cdots ,x_{i+n-1}), \cdots , x_{2n-1}) = (x_{1},\cdots, (x_{j},\cdots, x_{j-n+1}), \cdots, x_{2n-1})
$$
whenever $1 \leq i < j \leq n$. When $n=2$, this also reduces to the classical associativity. For the case when $n=3$, we have
$$
((x_{1},x_{2},x_{3}),x_{4},x_{5}) = (x_{1},(x_{2},x_{3},x_{4}),x_{5})= (x_{1},x_{2},(x_{3},x_{4},x_{5})).
$$

\begin{prop}
Let $\mathit{PAs}_{3}$ be the nonsymmetric operad of the partially associative $3$-algebra with product $\omega = (\cdot, \cdot, \cdot)$. Then an $\dsu(\mathit{PAs}_3)$-algebra, called a {\bf partially dendriform $3$-algebra}, is a vector space $A$ with three trilinear operations $\nwarrow, \uparrow, \nearrow$ such that
\begin{eqnarray*}
&\nwarrow (\nwarrow (x_{1},x_{2},x_{3}), x_{4},x_{5}) + \nwarrow (x_{1}, \ast (x_{2},x_{3},x_{4}), x_{5}) + \nwarrow (x_{1}, x_{2}, \ast (x_{3},x_{4},x_{5})) = 0,&    \\
&\nwarrow (\uparrow (x_{1},x_{2},x_{3}), x_{4},x_{5}) + \uparrow (x_{1}, \nwarrow (x_{2},x_{3},x_{4}), x_{5}) + \uparrow (x_{1}, x_{2}, \ast(x_{3},x_{4},x_{5}))=0,& \\
&\nwarrow (\nearrow (x_{1},x_{2},x_{3}), x_{4},x_{5}) + \uparrow (x_{1}, \uparrow (x_{2},x_{3},x_{4}), x_{5}) + \nearrow (x_{1}, x_{2},\nwarrow(x_{3},x_{4}, x_{5}))=0,& \\
&\uparrow (\ast (x_{1},x_{2},x_{3}), x_{4},x_{5})  +  \uparrow (x_{1}, \nearrow (x_{2},x_{3},x_{4}), x_{5}) + \nearrow (x_{1}, x_{2},\uparrow(x_{3},x_{4}, x_{5})) =0,& \\
&\nearrow (\ast (x_{1},x_{2},x_{3}), x_{4},x_{5}) + \nearrow (x_{1}, \ast (x_{2},x_{3},x_{4}), x_{5}) + \nearrow (x_{1}, x_{2},\nearrow(x_{3},x_{4}, x_{5}))=0.&
\end{eqnarray*}
Here we have used the notation $\ast = \nwarrow + \uparrow + \nearrow$.
\end{prop}
\begin{proof}
Let $r$ denote the relation in Eq.~(\mref{eq:pas}). By Proposition~\mref{pp:exp}, we have
\begin{eqnarray*}
\dsu_{x_{1}}  (r) &=&  \left\{ \svec{\omega}{e_{1}} (\svec{\omega}{e_{1}} (x_{1},x_{2},x_{3}), x_{4},x_{5}) + \svec{\omega}{e_{1}} (x_{1},\svec{\omega}{\ast}(x_{2},x_{3},x_{4}) ,x_{5}) + \svec{\omega}{e_{1}} (x_{1},x_{2},\svec{\omega}{\ast}(x_{3},x_{4},x_{5}))  \right\};\\
\dsu_{x_{2}} (r)   &=& \left\{ \svec{\omega}{e_{1}} (\svec{\omega}{e_{2}} (x_{1},x_{2},x_{3}), x_{4},x_{5}) + \svec{\omega}{e_{2}} (x_{1},\svec{\omega}{e_{1}}(x_{2},x_{3},x_{4}) ,x_{5}) + \svec{\omega}{e_{2}} (x_{1},x_{2},\svec{\omega}{\ast}(x_{3},x_{4},x_{5}))  \right\} ;\\
\dsu_{x_{3}}(r) &= &\left\{ \svec{\omega}{e_{1}} (\svec{\omega}{e_{3}} (x_{1},x_{2},x_{3}), x_{4},x_{5}) + \svec{\omega}{e_{2}} (x_{1},\svec{\omega}{e_{2}}(x_{2},x_{3},x_{4}) ,x_{5}) + \svec{\omega}{e_{3}} (x_{1},x_{2},\svec{\omega}{e_{1}}(x_{3},x_{4},x_{5}))  \right\} ;\\
\dsu_{x_{4}}  (r) &=&  \left\{ \svec{\omega}{e_{2}} (\svec{\omega}{\ast} (x_{1},x_{2},x_{3}), x_{4},x_{5}) + \svec{\omega}{e_{2}} (x_{1},\svec{\omega}{e_{3}}(x_{2},x_{3},x_{4}) ,x_{5}) + \svec{\omega}{e_{3}} (x_{1},x_{2},\svec{\omega}{e_{2}}(x_{3},x_{4},x_{5}))  \right\};\\
\dsu_{x_{5}} (r)  & =& \left\{ \svec{\omega}{e_{3}} (\svec{\omega}{\ast} (x_{1},x_{2},x_{3}), x_{4},x_{5}) + \svec{\omega}{e_{3}} (x_{1},\svec{\omega}{\ast}(x_{2},x_{3},x_{4}) ,x_{5}) + \svec{\omega}{e_{3}} (x_{1},x_{2},\svec{\omega}{e_{3}}(x_{3},x_{4},x_{5}))  \right\}.
\end{eqnarray*}
Then abbreviating $\nwarrow = \svec{\omega}{e_{1}}, \uparrow = \svec{\omega}{e_{2}}, \nearrow = \svec{\omega}{e_{3}}$, we obtain the relations
in the proposition.
\end{proof}

Similarly, on the level of operads, we have
\begin{prop}
Let $\mathit{TAs}_{3}$ be the nonsymmetric operad of the totally associative $3$-algebra with product $\omega = (\cdot, \cdot, \cdot)$. Then the operad $\dsu(\mathit{TAs}_3)$, called the {\bf totally dendriform $3$-operad}, has its arity graded space $V$ concentrated in $V_3=\bfk \{ \nwarrow, \uparrow, \nearrow\}$ and its relation space generated by
\begin{eqnarray*}
&\nwarrow \circ (\nwarrow \ot \id \ot \id) - \nwarrow \circ (\id \ot \ast \ot \id), \quad
\nwarrow \circ (\nwarrow \ot \id \ot \id) - \nwarrow  \circ (\id \ot \id  \ot \ast),&\\
&\nwarrow \circ (\uparrow \ot \id \ot \id) - \uparrow \circ ( \id \ot \nwarrow \ot \id), \quad
\nwarrow \circ (\uparrow \ot \id \ot \id) - \uparrow \circ (\id \ot \id \ot \ast),& \\
&\nwarrow \circ (\nearrow \ot \id \ot \id) - \uparrow \circ (\id \ot \uparrow \ot \id), \quad
\nwarrow \circ (\nearrow \ot \id \ot \id) - \nearrow \circ (\id \ot \id \ot \nwarrow),& \\
&\uparrow \circ (\ast \ot \id \ot \id)    -  \uparrow  \circ (\id \ot \nearrow \ot \id), \quad
\uparrow \circ (\ast \ot \id \ot \id)    - \nearrow \circ (\id \ot \id \ot \uparrow),& \\
&\nearrow \circ (\ast \ot \id \ot \id)    - \nearrow \circ (\id \ot \ast \ot \id), \quad
\nearrow \circ (\ast \ot \id \ot \id)    - \nearrow \circ (\id \ot \id \ot \nearrow).&
\end{eqnarray*}
Here we have used the notation $\ast = \nwarrow + \uparrow + \nearrow$.
\end{prop}

Furthermore, we can similarly use the arity-splitting of partially or totally associative $n$-algebra and give the notions of ``partially  or totally dendriform $n$-algebra".
We can also consider the power-splitting of the partially and totally associative $n$-algebra and give suitable extensions of tridendriform algebra in the context of $n$-algebras.

\subsubsection{The operad $Dend_\infty$ as the arity splitting of the operad $A_\infty$}
\mlabel{sss:ainf}

An $A_{\infty}$-algebra (or $Ass_{\infty}$-algebra)~\mcite{LV,St} is defined by Stasheff and has important applications in string theory. It has an $n$-ary generating operation  $\omega_{n}$ for every $n \geq 1$ that satisfy the relations
\begin{eqnarray}
&\omega_{1} \circ \omega_{1} =0,& \notag\\
& \partial(\omega_{n}) = \displaystyle \sum_{\tiny{ \begin{array}{l} n=p+q+r\\
k=p+1+r\\
k>1, q>1 \end{array}
}}(-1)^{p+qr}\omega_{k} \circ (id^{\ot p} \ot \omega_{q} \ot id^{\ot r}), \quad n \geq 2,& \mlabel{eq:reln}
\end{eqnarray}
where $
\partial(\omega_{n}) := \omega_{1} \circ \omega_{n} - (-1)^{n-2} \omega_{n} \circ \big(
(\omega_{1}, id,\cdots,id) + \cdots + (id, \cdots, id, \omega_{1}) \big)$  and $k = p + 1 + r$.

A $Dend_{\infty}$-algebra~\cite[\S 13.6.13]{LV} has $n$ $n$-ary generating operations $\omega_{n,i}, 1 \leq
i \leq n$, for each $n \geq 2$, that satisfy the relations
\begin{eqnarray}
&\omega_{1,1} \circ \omega_{1,1} =0,& \notag\\
&\partial(\omega_{n,i}) = \displaystyle \sum_{(p,q,r,\ell,j)} (-1)^{p+qr} \omega_{p+1+r,\ell}(id^{\ot p} \ot \omega_{q,j} \ot id^{\ot r}),& \mlabel{eq:dinf}
\end{eqnarray}
where, for fixed $n$ and $i$, the sum is extended to all the quintuples $(p, q, r, \ell, j)$ satisfying
$$p \geq 0, q \geq 2, r \geq 0, p+q+r = n, 1 \leq \ell \leq p+1+r, 1 \leq j \leq q$$
and the condition
$$\left\{\begin{array}{ll} i = q + \ell -1, & \text{when } 1 \leq p + 1 \leq \ell-1, \\
i = \ell-1 + j, & \text{when  } p + 1 = \ell, \\
i = \ell, &\text{when } \ell + 1 \leq p + 1.\end{array}\right.
$$
Note that the last condition is equivalent to
$$ \left \{\begin{array}{ll}
\ell = i-q+1, & \text{when } p+q+1 \leq i \leq n, \\
j=i-p, \ell=p+1,& \text{when } p+1 \leq i \leq p+q, \\
\ell=i, &1 \leq i \leq p.
\end{array}\right .
$$
For fixed $n$ and $i\in [n]$, by the definition of the arity-splitting and the abbreviation $\omega_{n,i} := \svec{\omega_{n}}{e_{i}}$, we have
$$
\dsu_{i}(\partial(\omega_{n})) = \omega_{1,1} \circ \omega_{n,i} - (-1)^{n-2} \omega_{n,i} \circ \big(
(\omega_{1,1}, id,\cdots,id) + \cdots + (id, \cdots, id, \omega_{1,1}) \big)= \partial (\omega_{n,i}).
$$
Also for any given triple $p,q,r$, we have
\begin{equation*}
\dsu_{i} \big( \omega_{k} \circ (id^{\ot p} \ot \omega_{q} \ot id^{\ot r})\big) = \left\{ \begin{array}{ll} \omega_{k,i} \circ (id^{\ot p} \ot \omega_{q,\ast} \ot id^{\ot r}), & \quad 1 \leq i \leq p,  \\  \omega_{k,p+1} \circ (id^{\ot p} \ot \omega_{q,i-p} \ot id^{\ot r}), & p+1 \leq i \leq p+q, \\ \omega_{k,i-q+1} \circ (id^{\ot p} \ot \omega_{q,\ast} \ot id^{\ot r}), &p+q+1 \leq i \leq n,
 \end{array} \right.
\end{equation*}
where $\displaystyle \omega_{q,\ast} = \sum_{j=1}^{q} \omega_{q,j}$.
Thus Eq.~(\mref{eq:dinf}) is just $\dsu$ applied to Eq.~(\mref{eq:reln}) and we obtain

\begin{prop}
$\dsu(A_{\infty}) = Dend_{\infty}$.
\end{prop}

\subsection{Examples of splittings of symmetric operads}
\mlabel{ss:egs}

We give some examples of splittings of (symmetric) operads. First note that
$$ \dsu(Lie)=\mathrm{BSu}(Lie)=pre\text{-}Lie,\quad
\tsu(Lie) = \mathrm{TSu}(Lie) =PostLie.$$ We next focus on operads that are not binary.

\subsubsection{$n$-Lie operads}
Recall that a $3$-Lie algebra is a vector space with a trilinear skew-symmetric operation $[\cdot, \cdot, \cdot]$ satisfies the $3$-Jacobi identity:
\begin{equation}
[[x_{1},x_{2},x_{3}], x_{4},x_{5}] = [[x_{1},x_{4},x_{5}],x_{2},x_{3}] + [x_{1}, [x_{2},x_{4},x_{5}],x_{3}] + [x_{1},x_{2},[x_{3},x_{4},x_{5}]].
\mlabel{eq:3lie}
\end{equation}
Let $3$-$Lie$ be the operad of the $3$-Lie algebra with product $\omega = [\cdot, \cdot, \cdot]$. Let $r$ denote the homogenous element from the  $3$-Jacobi identity from Eq.~(\mref{eq:3lie}). Then we have
\allowdisplaybreaks{
\begin{eqnarray*}
\dsu_{x_{1}} (r) &=&\svec{\omega}{e_{1}}\left( \svec{\omega}{e_{1}}(x_{1},x_{2},x_{3}), x_{4},x_{5} \right) - \svec{\omega}{e_{1}} \left(  \svec{\omega}{e_{1}} (x_{1},x_{4},x_{5}),x_{2},x_{3} \right) \\
&&- \svec{\omega}{e_{1}} \left(x_{1}, \svec{\omega}{\ast} (x_{2},x_{4},x_{5}), x_{3} \right) -
\svec{\omega}{e_{1}} \left( x_{1},x_{2}, \svec{\omega}{\ast}(x_{3},x_{4},x_{5}) \right),\\
\dsu_{x_{4}} (r) &=& \svec{\omega}{e_{2}}\left( \svec{\omega}{\ast}(x_{1},x_{2},x_{3}), x_{4},x_{5} \right) - \svec{\omega}{e_{1}} \left(  \svec{\omega}{e_{2}} (x_{1},x_{4},x_{5}),x_{2},x_{3} \right) \\
&&- \svec{\omega}{e_{2}} \left(x_{1}, \svec{\omega}{e_{2}} (x_{2},x_{4},x_{5}), x_{3} \right) -
\svec{\omega}{e_{3}} \left( x_{1},x_{2}, \svec{\omega}{e_{2}}(x_{3},x_{4},x_{5}) \right).
\end{eqnarray*}}
Similar computations apply to $\dsu_{x_{2}}, \dsu_{x_{3}}$ and $\dsu_{x_{5}}$. However these relations can also be obtained from the relations of $\dsu_{x_{1}}$ and $\dsu_{x_{4}}$ by a permutation of the variables. Replace the operation $\svec{\omega}{e_{1}}$ by $\{\cdot, \cdot, \cdot\}$. The group actions
$$
\svec{\omega}{e_{1}}^{(23)} = \svec{\omega^{(23)}}{e_{1}} = - \svec{\omega} {e_{1}},
$$
show that $\{\cdot, \cdot, \cdot\}$ satisfies the local-skew symmetry relation. Furthermore

$$
\svec{\omega}{e_{2}}^{(12)} = -\svec{\omega}{e_{1}} , \quad \svec{\omega}{e_{3}}^{(13)} = -\svec{\omega}{e_{1}}.
$$
Since the arity-splitting (that is, bisuccessor) $\dsu(Lie)$ of the operad $Lie$ of the Lie algebra is the operad of the pre-Lie algebra, it is natural to use $\dsu$(3-$Lie$) to give the the following definition.
\begin{defn}\mlabel{defn:3prelie}
{\rm
A {\bf 3-pre-Lie algebra} is a vector space $A$ with a trilinear map $\{\cdot,\cdot,\cdot\} : A^{\ot 3} \longrightarrow A$ such that
\begin{equation}
\{x_{1},x_{2},x_{3}\} = -\{x_{1},x_{3},x_{2}\},
\end{equation}
\begin{eqnarray}
&\{\{x_{1},x_{2},x_{3}\},x_{4},x_{5}\} = \{\{x_{1},x_{4},x_{5}\}, x_{2},x_{3}\} + \{x_{1},\bigcirc \{x_{2},x_{4},x_{5}\},x_{3}\} + \{x_{1},x_{2}, \bigcirc \{x_{3},x_{4},x_{5}\}\},&
\end{eqnarray}
\begin{eqnarray}
&\{x_{4}, \bigcirc\{x_{1},x_{2},x_{3}\}, x_{5}\} = \{\{x_{4},x_{1},x_{5}\},x_{2},x_{3}\}  + \{\{x_{4},x_{2},x_{5}\},x_{3},x_{1}\}  + \{\{x_{4},x_{3},x_{5}\},x_{1},x_{2}\},&
\end{eqnarray}
where $\bigcirc\{x,y,z\} = \{x,y,z\} + \{y,z,x\} + \{z,x,y\}$.
}
\end{defn}

In general, an {\bf $n$-Lie algebra} is a vector space over a field $\bfk$ endowed with an $n$-ary multi-linear skew-symmetric operation $[\cdot, \cdots,\cdot]$ satisfying the $n$-Jacobi identity
\begin{equation}
[[x_{1}, \cdots, x_{n}], x_{n+1}, \cdots, x_{2n-1}] = \sum_{i=1}^{n}[x_{1},\cdots, [x_{i},x_{n+1},\cdots,x_{2n-1}], \cdots, x_{n}].
\end{equation}

Computing the arity-splitting of $n$-Lie and replacing the operation $\svec{\omega}{e_{1}}$ by $\{\cdot, \cdots, \cdot\}$, we have

\begin{defn}\mlabel{defn:nprelie}
{\rm An {\bf $n$-pre-Lie algebra} is a vector space $A$ with a $n$-linear map $\{\cdot,\cdots,\cdot\} : A^{\ot n} \longrightarrow A$ such that
\begin{eqnarray}
\{x_{1},x_{2},\cdots, x_{n}\} &=& {\rm sgn} (\sigma) \{x_{1}, x_{\sigma(2)}, x_{\sigma(3)},\cdots, x_{\sigma(n)}\}, \mbox{where}~ \sigma \in \BS_{n} ~\mbox{and} ~\sigma(1) =1,
\end{eqnarray}
\begin{eqnarray}
\{\{x_{1}, \cdots, x_{n}\}, x_{n+1}, \cdots, x_{2n-1}\} &=&\{\{x_{1}, x_{n+1}, \cdots, x_{2n-1}\}, x_{2}, \cdots, x_{n}\} \notag\\
&& +\sum_{i=2}^{n} \{x_{1},\cdots, \bigcirc\{x_{i},x_{n+1},\cdots,x_{2n-1}\}, \cdots, x_{n}\},
\end{eqnarray}
\begin{eqnarray}
\lefteqn{\{x_{n+1},\bigcirc\{x_{1},\cdots, x_{n}\}, x_{n+2}, \cdots, x_{2n-1}\}} \notag\\
&=&\displaystyle (-1)^{(1+i)(n+1-i)}\sum_{i=1}^{n}  \{\{x_{n+1},x_{i},x_{n+2},\cdots,x_{2n-1}\}, x_{i+1}, x_{i+2}, \cdots, x_{n}, x_{1}, x_{2}, \cdots, x_{i-1}\},
\end{eqnarray}
where $\displaystyle \bigcirc\{x_{1},x_{2},\cdots,x_{n}\} = (-1)^{(1+i)(n+1-i)} \sum_{i=1}^{n}\{x_{i},x_{i+1},x_{i+2}, \cdots, x_{n}, x_{1},x_{2}, \cdots, x_{i-1}\}$.
}
\end{defn}

In fact, there is another case of $3$-ary operation that is closely related to 3-Lie algebras. A {\bf generalized Lie algebra of order 3 or Lie-3 algebra} \cite{BM,BM1,GGR} is a vector space $A$ together with a trilinear skew-symmetric operation $[\cdot,\cdot,\cdot]$ such that
\begin{eqnarray*}
&&[[x_1,x_2,x_3],x_4, x_5]-[[x_1,x_2,x_4],x_3,x_5]+[[x_1,x_3,x_4],x_2,x_5]-[[x_2,x_3,x_4],x_1,x_5]\\
&&+[[x_1,x_2,x_5],x_3,x_4]+[[x_3,x_4,x_5],x_1,x_2]-[[x_1,x_3,x_5],x_2,x_4]
-[[x_2,x_4,x_5], x_1,x_3]\\
&&+[[x_1,x_4,x_5],x_2,x_3]+[[x_2,x_3,x_5],x_1,x_4]=0.
\end{eqnarray*}
It is known~\cite{BM1} that a 3-Lie algebra is a generalized Lie algebra of order 3.

\begin{remark}\mlabel{rem:paslie}
{\rm For an $n$-algebra $A$ and its operation $\omega$, the commutator of $A$ is
$$
\sum_{\sigma \in \BS_{n}} {\rm sgn}(\sigma) \omega(x_{\sigma(1)}, \cdots, x_{\sigma(i)}, \cdots, x_{\sigma(n)}).
$$
A partially associative $3$-algebra $A$ is a 3-Lie admissible algebra, more precisely, the commutator of $A$ makes $A$ into a generalized Lie algebra of order $3$ \mcite{GGR}.
}
\end{remark}

Similarly to the $3$-Lie algebra, we use the arity-splitting of the generalized Lie algebra of order 3 to give the following notion:
\begin{defn}\mlabel{defn:gprelie}
{\rm A {\bf generalized pre-Lie algebra of order $3$} is a vector space $A$ with a trilinear map $\{\cdot,\cdot,\cdot\} : A^{\ot 3} \longrightarrow A$ such that
\begin{equation}
\{x_{1},x_{2},x_{3}\} = -\{x_{1},x_{3},x_{2}\},
\end{equation}
\begin{eqnarray}
\{\{x_{1},x_{2},x_{3}\},x_{4},x_{5}\} &=& \{\{x_{1},x_{2},x_{4}\}, x_{3},x_{5}\} - \{\{x_{1},x_{2},x_{5}\}, x_{3},x_{4}\} - \{\{x_{1},x_{3},x_{4}\}, x_{2},x_{5}\} \notag\\
&+& \{\{x_{1},x_{3},x_{5}\}, x_{2},x_{4}\} - \{\{x_{1},x_{4},x_{5}\}, x_{2},x_{3}\} - \{x_{1}, \bigcirc\{x_{2},x_{3},x_{4}\},x_{5}\} \\
&+& \{x_{1}, \bigcirc\{x_{2},x_{3},x_{5}\},x_{4}\} - \{x_{1}, \bigcirc\{x_{2},x_{4},x_{5}\},x_{3}\} + \{x_{1}, \bigcirc \{x_{3},x_{4},x_{5}\},x_{2}\}\notag
\end{eqnarray}
where $\bigcirc \{x,y,z\} = \{x,y,z\} + \{y,z,x\} + \{z,x,y\}$.
}
\end{defn}

\begin{prop}\label{pp:localcom}
Define the local commutator of a partially associative $3$-algebra $A$ by
$$
\{x,y,z\} = (x,y,z) - (x,z,y).
$$
Then $(A, \{\cdot, \cdot, \cdot\})$ is a generalized pre-Lie algebra of order $3$.
\end{prop}

\begin{proof}
The proof is a straightforward computation.
\end{proof}

See Proposition~\mref{prop:inclusion} for the relationship between the $3$-pre-Lie algebra and generalized pre-Lie algebra of order 3.

\subsubsection{The operad $PL_\infty$ as the arity splitting of the operad $L_\infty$}
\mlabel{sss:linf}

An {\bf $L_{\infty}$-algebra} (or {\bf $Lie_{\infty}$-algebra}) \mcite{LM,LV}on a graded vector space L with a system $\{m_{n}~|~ n \geq1 \}$ of linear maps $m_{n}:L^{\ot n}  \longrightarrow  L$ with $\deg(m_{n}) = n-2$ that are antisymmetric in the sense that
$$
m_{n}(x_{\sigma(1)}, \cdots, x_{\sigma(n)}) = sgn(\sigma)m_{n}(x_{1},\cdots ,x_{n}), \quad \text{for all } \sigma \in \mathbb{S}_{n}, x_{1},\cdots, x_{n} \in L,
$$
and satisfy the following generalized form of the Jacobi
identity:
\begin{equation}
R_{L} :=\sum_{i+j=n+1}\sum_{\sigma \in Sh_{i,n-i}} \epsilon(\sigma, \bar{v}) m_{j} (m_{i} (x_{\sigma(1)},\cdots, x_{\sigma(i)}), x_{\sigma(i+1)},\cdots, x_{\sigma(n)}) = 0.
\mlabel{eq:rl}
\end{equation}
Here $Sh_{i,n-i}\subseteq \BS_n$ is the set of $(i,n-i)$-shuffles.

A {\bf $PL_{\infty}$} (or {\bf $pre$-$Lie_{\infty}$-algebra}) \mcite{Ch1} on a graded vector space $V$ is a system of linear maps
$\ell_{n}: V \ot S^{n}(V) \longrightarrow V$ of degree $\ell_{n} = n -1$, $n \geq 1$, that satisfy
\begin{eqnarray}
\lefteqn{\sum_{i+j=n}\sum_{\sigma \in Sh_{i,n-i}} \epsilon(\sigma, \bar{v}) \ell_{j}(\ell_{i}(v_{0} \ot v_{\sigma(1)},\cdots, v_{\sigma(i)}), v_{\sigma(i+1)}, \cdots, v_{\sigma(n)})} \notag\\
&& + (-1)^{|v_{0}||\ell_{i}|} \sum_{i+j=n, j \geq 1}\sum_{\sigma \in Sh_{1,i,n-i-1}} \epsilon(\sigma, \bar{v})\ell_{j}(v_{0} \ot \ell_{i}(v_{\sigma(1)}, \cdots, v_{\sigma(i+1)}), v_{\sigma(i+2)}, \cdots, v_{\sigma(n)}) = 0, \mlabel{eq:ch}
\end{eqnarray}
where $S(V)=\bigoplus_{n\geq 0} S^n(V)$ is the graded symmetric algebra generated by $V$ and $\epsilon(\sigma, \bar{v})$ is the Koszul sign.

We next relate $PL_\infty$ to $\dsu(L_\infty)$.
For any $n\geq 1$, use the abbreviation $m_{n,i} = \svec{m_{n}}{e_{i}}$. By the $\BS_n$-action on $\svec{m_n}{e_i}$, we have
$$
m_{n,1} (x_{1},x_{\sigma(2)}, x_{\sigma(3)},\cdots, x_{\sigma(n)}) = sgn(\sigma) m_{n,1}(x_{1}, x_{2},\cdots, x_{n}) , \sigma \in \mathbb{S}_{n}, \sigma(1) =1,
$$
and thus $m_{n,1}$ can be regarded as a linear map from $V \ot S^{n-1}(V)$ to $V$.

For $i \neq 1$, we have
$$
m_{n,i} = {\rm sgn} ((1i)) m_{n,1}^{(1i)} = - m_{n,1}^{(1i)}.
$$
Fixed $n$ and $1\leq i\leq n-1$, there are $|Sh_{i,n-i}| = C_{n}^{i}$ $(i,n-i)$-shuffles. These shuffles can be divided into two subsets:
\begin{enumerate}
\item $Sh_{i,n-i}^{1} :=\{\sigma\in Sh_{i,n-i} ~|~ \sigma(1) =1\}$,  \quad   $|Sh_{i,n-i}^{1}| = C_{n-1}^{i-1}$;
\item $Sh_{i,n-i}^{2} :=\{\sigma\in Sh_{i,n-i} ~|~ \sigma(i+1) =1\}$,  \quad   $|Sh_{i,n-i}^{2}| = C_{n-1}^{i}$.
\end{enumerate}

For each $\sigma \in Sh_{i,n-i}^{1}$, we have
$$
\dsu_{x_{1}} \big( m_{j} (m_{i} (x_{\sigma(1)},\cdots, x_{\sigma(i)}), x_{\sigma(i+1)},\cdots, x_{\sigma(n)}) \big)= m_{j,1} (m_{i,1} (x_{1},x_{\sigma(2)}\cdots, x_{\sigma(i)}), x_{\sigma(i+1)},\cdots, x_{\sigma(n)}).
$$
Then we have
\begin{eqnarray*}
\lefteqn{\sum_{\sigma \in Sh_{i,n-i}^{1}} \epsilon(\sigma, \bar{v})  m_{j,1} (m_{i,1} (x_{1},x_{\sigma(2)}\cdots, x_{\sigma(i)}), x_{\sigma(i+1)},\cdots, x_{\sigma(n)})}\\
&=& \sum_{\sigma \in Sh_{i-1,n-i}} \epsilon(\sigma, \bar{v})  m_{j,1} (m_{i,1} (x_{1},x_{\sigma(2)}\cdots, x_{\sigma(i)}), x_{\sigma(i+1)},\cdots, x_{\sigma(n)}),
\end{eqnarray*}
where $Sh_{i-1,n-i}$ is regarded as the set of $(i-1,n-i)$-shuffles on the set of $\{2,\cdots,n\}$.

For each $\sigma \in Sh_{i,n-i}^{2}$, we have
$$
\dsu_{x_{1}} \big( m_{j} ( m_{i} (x_{\sigma(1)},\cdots, x_{\sigma(i)}), x_{\sigma(i+1)},\cdots, x_{\sigma(n)}) \big)= m_{j,2} (m_{i,\star} (x_{\sigma(1)},\cdots, x_{\sigma(i)}), x_{1},x_{\sigma(i+2)},\cdots, x_{\sigma(n)}),
$$
where $m_{i,\star} = \sum_{s=1}^{i} m_{i,s}$.
We also have
\begin{eqnarray*}
\lefteqn{m_{j,2} ( m_{i,s} (x_{\sigma(1)},\cdots, x_{\sigma(i)}), x_{1},x_{\sigma(i+2)},\cdots, x_{\sigma(n)})}\\
&=& (-1)^{|x_{1}||m_{i,s}|} m_{j,1}(x_{1}, m_{i,s} (x_{\sigma(1)},\cdots, x_{\sigma(i)}), x_{\sigma(i+2)},\cdots, x_{\sigma(n)}),
\end{eqnarray*}
and $(-1)^{|x_{1}||m_{i,s}|} = (-1)^{|x_{1}||m_{i,1}|}$ where
$$
m_{i,s}(y_{1},y_{2},\cdots,y_{i}) = (-1)^{s-1} m_{i,1} (y_{s},y_{1},y_{2},\cdots,y_{s-1},y_{s+1},y_{s+2},\cdots, y_{i}).
$$
Thus we obtain
$$
m_{i,\star} (y_{1},y_{2},\cdots,y_{i}) = \sum_{\tau \in Sh_{1, i-1}}  \epsilon(\tau, \bar{v}) m_{i,1} (y_{\tau(1)},\cdots,y_{\tau(i)}).
$$

\begin{lemma}
The map
$$ \Gamma: Sh_{i,n-i} \times Sh_{1,i-1} \longrightarrow Sh_{1,i-1,n-i}, \quad
\Gamma \big(\sigma, \tau\big)(j)  = \left\{ \begin{array}{ll} \sigma(\tau(j)),& 1 \leq j \leq i,\\ \sigma(j), & j \geq i+1,  \end{array} \right.
$$
is a bijection.
\mlabel{lem:bij}
\end{lemma}

\begin{proof}
This map is injective since if $\Gamma(\sigma,\tau)=\Gamma(\sigma',\tau')$, then $\sigma(j)=\sigma'(j)$ for $j\geq i+1$ which implies that $\sigma=\sigma'$ and then $\tau=\tau'$. Then the map must be a bijection since the cardinality of the domain and codomain are the same.
\end{proof}

By Lemma~\mref{lem:bij}, we have

\begin{eqnarray*}
\lefteqn{ \sum_{\sigma \in Sh_{i,n-i}^{2}} \epsilon(\sigma, \bar{v}) m_{j,1}(x_{1}, m_{i,\star} (x_{\sigma(1)},\cdots, x_{\sigma(i)}), x_{\sigma(i+2)},\cdots, x_{\sigma(n)})}\\
&=& \sum_{\sigma \in Sh_{i,n-i}^{2}} \sum_{\tau\in Sh_{1,i-1}} \epsilon(\sigma, \bar{v}) \epsilon(\tau, \bar{v})m_{j,1}(x_{1}, m_{i,1} (x_{\sigma(\tau(1))},\cdots, x_{\sigma(\tau(i))}), x_{\sigma(i+2)},\cdots, x_{\sigma(n)}) \\
&=& \sum_{\sigma \in Sh_{1,i-1,n-i-1}}  \epsilon(\sigma, \bar{v})  m_{j,1}(x_{1}, m_{i,1} (x_{\sigma(1)},\cdots, x_{\sigma(i)}), x_{\sigma(i+2)},\cdots, x_{\sigma(n)}).
\end{eqnarray*}
Then we have

\begin{eqnarray*}
\lefteqn{\dsu_{x_{1}} \big(  \sum_{\sigma \in Sh_{i,n-i}} \epsilon(\sigma, \bar{v}) m_{j} ( m_{i} (x_{\sigma(1)},\cdots, x_{\sigma(i)}), x_{\sigma(i+1)},\cdots, x_{\sigma(n)}) \big)}\\
&=&  \sum_{\sigma \in Sh_{i,n-i}^{1}}  \epsilon(\sigma, \bar{v}) m_{j,1} ( m_{i,1} (x_{1},x_{\sigma(2)},\cdots, x_{\sigma(i)}), x_{\sigma(i+1)},\cdots, x_{\sigma(n)}) \\
&&+ \sum_{\sigma \in Sh_{i,n-i}^{2}}   \epsilon(\sigma, \bar{v}) m_{j,1}(x_{1}, m_{i,\star} (x_{\sigma(1)},\cdots, x_{\sigma(i)}), x_{\sigma(i+2)},\cdots, x_{\sigma(n)})\\
&=& \sum_{\sigma \in Sh_{i-1,n-i}}   \epsilon(\sigma, \bar{v}) m_{j,1} ( m_{i,1} (x_{1},x_{\sigma(2)},\cdots, x_{\sigma(i)}), x_{\sigma(i+1)},\cdots, x_{\sigma(n)}) \\
&&+(-1)^{|x_{1}||m_{i,1}|} \sum_{\sigma \in Sh_{1,i-1,n-i-1}}   \epsilon(\sigma, \bar{v}) m_{j,1}(x_{1}, m_{i,1} (x_{\sigma(1)},\cdots, x_{\sigma(i)}), x_{\sigma(i+2)},\cdots, x_{\sigma(n)})\\
&=& \sum_{\sigma \in Sh_{i-1,n-i}}   \epsilon(\sigma, \bar{v}) \ell_{n-i} ( \ell_{i-1} (x_{1},x_{\sigma(2)},\cdots, x_{\sigma(i)}), x_{\sigma(i+1)},\cdots, x_{\sigma(n)}) \\
&&+(-1)^{|x_{1}||\ell_{i-1}|} \sum_{\sigma \in Sh_{1,i-1,n-i-1}}   \epsilon(\sigma, \bar{v}) \ell_{n-i}(x_{1}, \ell_{i-1} (x_{\sigma(1)},\cdots, x_{\sigma(i)}), x_{\sigma(i+2)},\cdots, x_{\sigma(n)}),
\end{eqnarray*}
where the last equation is obtained by replacing $m_{i,1}$ with $\ell_{i-1}$ and $m_{j,1}$ with $\ell_{j-1} = \ell_{n+1-i-1} = \ell_{n-i}$ and $ j \geq 1$.
Hence, for $R_L$ defined in Eq.~(\mref{eq:rl}), we have
\begin{eqnarray*}
\dsu_{x_{1}} (R_{L}) &=& \sum_{i+j=n+1}\sum_{\sigma \in Sh_{i-1,n-i}}   \epsilon(\sigma, \bar{v}) \ell_{n-i} ( \ell_{i-1} (x_{1},x_{\sigma(2)},\cdots, x_{\sigma(i)}), x_{\sigma(i+1)},\cdots, x_{\sigma(n)}) \\
&&+(-1)^{|x_{1}||\ell_{i-1}|} \sum_{i+j=n+1,j \geq 1}\sum_{\sigma \in Sh_{1,i-1,n-i-1}}   \epsilon(\sigma, \bar{v}) \ell_{n-i}(x_{1}, \ell_{i-1} (x_{\sigma(1)},\cdots, x_{\sigma(i)}), x_{\sigma(i+2)},\cdots, x_{\sigma(n)})\\
&=& \sum_{s+t= n-1} \sum_{\sigma \in Sh_{s,(n-1)-s}} \epsilon(\sigma, \bar{v}) \ell_{t} ( \ell_{s} (x_{1},x_{\sigma(2)},\cdots, x_{\sigma(s)},x_{\sigma(s+1)}), x_{\sigma(s+2)},\cdots, x_{\sigma(n)}) \\
&&+(-1)^{|x_{1}||\ell_{s}|} \sum_{s+t=n-1, t \geq 1}\sum_{\sigma \in Sh_{1,s,(n-1)-s-1}}   \epsilon(\sigma, \bar{v}) \ell_{t}(x_{1}, \ell_{s} (x_{\sigma(1)},\cdots, x_{\sigma(s+1)}), x_{\sigma(s+3)},\cdots, x_{\sigma(n)}).
\end{eqnarray*}
This agrees with the relations of $PL_{\infty}$ in Eq.~(\mref{eq:ch}). Note that $\dsu_{x_{p}}(R_L), 2\leq p\leq n$ can be obtained from $\dsu_{x_{1}}(R_L)$ by a permutation of variables. Therefore, we have

\begin{prop}
$\dsu(L_{\infty}) = PL_{\infty}$.
\end{prop}

\section{The splitting property and functorial property}
\mlabel{sec:spl}
In this section we prove the splitting property of \spl-splitting of an operad and thus justify the term \spl-splitting. We also prove that the process of \spl-splitting is compatible with some morphisms between operads.

\subsection{The splitting property}
We recall the following splitting property of the dendriform algebra that is simple yet fundamental in motivating all the subsequent studies of dendriform type or Loday algebras.

\begin{prop}
\mcite{Lo2} Let $(A,\prec,\succ)$ be a dendriform algebra. Then the operation on $A$ defined by $x \ast y:=x\prec y +x\succ y$ is associative.
\mlabel{pp:loday}
\end{prop}
On the operad level, this is interpreted as an operad morphism
\begin{equation}
Asso \to Dend, \quad \ast \mapsto \,\prec +\succ,
\mlabel{eq:dslp}
\end{equation}
from the operad $Asso$ of the associative algebra to the operad $Dend$ of the dendriform algebra. It is in this sense that the operations $\prec$ and $\succ$ give a splitting of the associative product $\ast$. This property has been generalized to many binary operads over the years and then eventually to all binary operads~\mcite{BBGN}. We will further generalize this property to the $\confm$-splitting of any operad.

\begin{lemma}\mlabel{lem:arity}
Let $\gensp=\bigoplus\limits_{n\geq 1}V_n$ be an arity graded vector space with basis $\genbas=\coprod_{n\geq 1}\genbas_{n}$ and let $\confm$ be either $\cala$ or the trivial \conf in Example~\mref{ex:conf}. For a labeled planar $n$-tree $\tau \in \calt(\genbas)$, we have the following equation in $\calt(V)$:
\begin{eqnarray}
\sum_{J \in \confn_{|\lin(\tau)|}} \ssp_{J}(\tau) = \ssp(\tau). \mlabel{eq:bsuall}
\end{eqnarray}
\mlabel{lem:all}
\end{lemma}

\begin{proof}
It is obvious that Eq.~(\mref{eq:bsuall}) holds when $\confm$ is the trivial \conf. So we only need to prove Eq. (\mref{eq:bsuall}) by induction on $|\lin(\tau)|$ for the \conf $\cala=(A_n)$. When $|\lin(\tau)| =1$, we have
$$
\sum_{x \in \lin(\tau)} \dsu_{x}(\tau) = \tau = \dsu(\tau).
$$
Now assume that Eq. (\ref{eq:bsuall}) holds for all $\tau \in \calt(\genbas)$ with $\lin(\tau) \leq n-1, n > 1,$ and consider an $n$-tree $\tau$ in $\calt(\genbas)$. Since $\tau = \omega(\tau_{1} \vee \tau_{2} \vee \cdots \vee \tau_{\ell})$ for some integer $\ell$ and $\omega \in \genbas(\ell)$, by the definition of the \spl-splitting of a planar tree and the induction hypothesis, we have
\begin{eqnarray*}
\lefteqn{\sum_{x \in \lin(\tau)} \dsu_{x}(\tau)}\\ &=&\sum_{i=1}^{\ell} \sum_{x \in \lin(\tau_{i})} \svec{\omega}{e_{i}}(~\dsu(\tau_{1}) \vee \cdots \vee \dsu(\tau_{i-1})  \vee \dsu_{x}(\tau_{i}) \vee \dsu(\tau_{i+1}) \vee \cdots \vee \dsu(\tau_{\ell}))\\
&=& \sum_{i=1}^{\ell} \svec{\omega}{e_{i}} (~\dsu(\tau_{1}) \vee \cdots \vee \dsu(\tau_{i-1}) \vee \Big(\sum_{x \in \lin(\tau_{i})}\dsu_{x}(\tau_{i})\Big) \vee \dsu(\tau_{i+1}) \vee \cdots \vee \dsu(\tau_{\ell}))\\
&=& \svec{\omega}{\ast}(~\dsu(\tau_{1}) \vee \cdots \vee \dsu(\tau_{i-1}) \vee \dsu(\tau_{i}) \vee \dsu(\tau_{i+1}) \vee \cdots \vee \dsu(\tau_{\ell}))\\
&=& \dsu(\tau).
\end{eqnarray*}
This completes the induction.
\end{proof}

\begin{prop}\mlabel{prop:general}
Let $\confm$ be an $\BS$-invariant \conf of index $m$ (see Eq.~(\mref{eq:ind})).
Then for any labeled planar tree $\tau\in \calt(\genbas)$ with $\lin(\tau)\leq m$, we have
\begin{equation}\mlabel{eq:gall}
\sum_{J\in \confn_{|\lin(\tau)|}} \ssp_J(\tau)=\ssp(\tau).
\end{equation}
Especially, when $m =\infty$ (the power-splitting case), Eq. (\mref{eq:gall}) holds for all labeled planar trees.
\end{prop}
\begin{proof}
We prove by induction on $|\lin(\tau)|$ for $|\lin(\tau)| \leq m$. Then $\ssp_{J}(\tau) = \tsu_{J}(\tau)$ and $\ssp(\tau) = \tsu(\tau)$.  When $|\lin(\tau)| =1$, by definition we have
$$
\sum_{x \in \lin(\tau)} \tsu_{x}(\tau) = \tau = \tsu(\tau).
$$
Now assume that Eq. (\ref{eq:gall}) holds for all $\tau \in \calt(\genbas)$ with $\lin(\tau) \leq n-1 $ for an $1< n \leq m$ and consider an $n$-tree $\tau$ in $\calt(\genbas)$. Let $\tau = \omega(\tau_{1} \vee \tau_{2} \vee \cdots \vee \tau_{\ell})$ for some integer $\ell$ and $\omega \in \genbas(\ell)$.
Let $\lin(\tau)=\{1,2,\cdots,n\}$ and $\lin(\tau_{p}) = \{k_{p-1}+1, \cdots, k_{p}\}$ with the convention that $k_{0} =0$ and $k_{\ell} =n$.
Define a map
$$\varphi: B_{|\lin(\tau)|} = \{J ~|~ \emptyset \neq J \subseteq \lin(\tau)\}\longrightarrow B_{\ell}, \quad \varphi(J) = J\sqcap \omega.
$$
Since for each $I = \{i_{1},i_{2},\cdots, i_{t}\}$, the image of $\{k_{i_{1}}, k_{i_{2}}, \cdots, k_{i_{t}}\} \in B_{|\lin(\tau)|}$ under $\varphi$ is $I$, we see that $\varphi$ is surjective and $\bigsqcup_{I \in B_{\ell}} \varphi^{-1} (I) = B_{|\lin(\tau)|}$. Thus

\begin{eqnarray*}
\sum_{J \in B_{|\lin(\tau)|}} \tsu(\tau) &=& \sum_{I \in B_{\ell}}  \sum_{J \in \varphi^{-1}(I)} \svec{\omega}{e_{I}}(~\tsu_{J\cap\lin(\tau_1)}(\tau_{1}) \vee \cdots  \vee \tsu_{J\cap\lin(\tau_\ell)}(\tau_{\ell})~).
\end{eqnarray*}

For a fixed $I = \{i_{1},i_{2},\cdots, i_{t}\}$, we have
\begin{equation}\mlabel{eq:inverse}
\varphi^{-1}(I) = \{J \subseteq \lin(\tau)~|~   J = J_{i_{1}} \sqcup J_{i_{2}} \sqcup \cdots \cup J_{i_{t}}, \emptyset \neq J_{i_{j}} \subseteq \lin(\tau_{i_{j}}), 1 \leq j \leq t\}
\end{equation}
which is in bijection with
$ \{\emptyset\neq J_{i_j}\subseteq \lin(\tau_{i_j})\} \times \cdots \times \{\emptyset\neq J_{i_j}\subseteq \lin(\tau_{i_j})\}.$ Thus by the induction hypothesis, we have
\begin{eqnarray*}
\lefteqn{\sum_{J \in \varphi^{-1}(I)} \svec{\omega}{e_{I}}(~\tsu_{J\cap\lin(\tau_1)}(\tau_{1}) \vee \cdots  \vee \tsu_{J\cap\lin(\tau_\ell)}(\tau_{\ell})~)}\\
&=& \sum_{J \in \varphi^{-1}(I)} \svec{\omega}{e_{I}}(~\tsu(\tau_{1}) \vee \cdots  \vee \tsu(\tau_{i_{1} -1}) \vee \tsu_{J\cap \lin(\tau_{i_1})}(\tau_{i_{1}}) \vee \tsu(\tau_{i_{1}+1}) \vee \cdots \vee \tsu(\tau_{i_{2}-1}) \vee \\
&& \hspace{-1cm}\tsu_{J \cap \lin(\tau_{i_{2}})}(\tau_{i_{2}})  \vee \tsu(\tau_{i_{2}+1}) \vee \cdots \vee   \tsu(\tau_{i_{t}-1}) \vee   \tsu_{J \cap \lin(\tau_{i_{t}})}(\tau_{i_{t}})  \vee \tsu(\tau_{i_{t}+1})\vee\cdots \vee \tsu(\tau_{\ell})~)\\
&=& \svec{\omega}{e_{I}}\left(~
\tsu(\tau_{1}) \vee \cdots  \vee \tsu(\tau_{i_{1} -1}) \vee \sum_{\emptyset \neq J_{j_1}\subseteq \lin(\tau_{i_1})}\tsu_{J_{j_1}}(\tau_{i_{1}}) \vee \tsu(\tau_{i_{1}+1})  \right .\\
&& \qquad  \vee \cdots \vee \tsu(\tau_{i_{2}-1}) \vee
\sum_{\emptyset \neq J_{j_2}\subseteq \lin(\tau_{i_2})}\tsu_{J_{j_2}}(\tau_{i_{2}}) \vee \tsu(\tau_{i_{2}+1}) \\
&&\qquad \left .\vee\cdots  \vee   \tsu(\tau_{i_{t}-1}) \vee  \sum_{\emptyset \neq J_{j_t}\subseteq \lin(\tau_{i_t})}\tsu_{J_{j_t}}(\tau_{i_{t}}) \vee \tsu(\tau_{i_{t}+1}) \vee\cdots \vee \tsu(\tau_{\ell})~\right)\\
&=&\svec{\omega}{e_{I}}(~\tsu(\tau_{1}) \vee \cdots  \vee \tsu(\tau_{\ell})~).
\end{eqnarray*}
Hence
\begin{eqnarray*}
\sum_{J \in B_{|\lin(\tau)|}} \tsu(\tau) &=& \sum_{I \in B_{\ell}} \svec{\omega}{e_{I}}(~\tsu(\tau_{1}) \vee \cdots  \vee \tsu(\tau_{\ell})~) = \tsu(\tau),
\end{eqnarray*}
completing the induction.
\end{proof}

\begin{coro}\mlabel{coro:special}
Let $\opd=\mathcal{T}(\gensp)/ (\relsp)$ be an operad with locally homogeneous relations
\begin{equation*}
R=\{r_s:=\sum_i c_{s,i}\tau_{s,i}, \ c_{s,i}\in\bfk, \tau_{s,i} \in \bigcup_{t\in \mathfrak{R}} t(\genbas), \ 1\leq s\leq k\}.
\end{equation*}
Let $\confm$ be a \conf with index $m$. If $\max \{|\lin(r_{s})|\}_{s} \leq m$, then
\begin{equation*}
\sum_{J\in \confn_{|\lin(r_{s})|}} \ssp_J(r_{s})=\ssp(r_{s}), \quad 1\leq i\leq k.
\end{equation*}
\end{coro}

The following result gives the precise meaning of splitting an operad, generalizing the splitting property of the associativity in Eq.~(\mref{eq:dslp})~\mcite{Lo2} and the splitting of a binary operad in~\mcite{BBGN}. For an operad $\calq=\calt(W)/(R_\calq)$, let $i_W:W\to \calp(W)$ and $p_W: \calt(W)\to \calq$ denote the natural injection and projection.

\begin{theorem}\mlabel{pp:suast}
Let $\calp = \calt(V) / (R)$ be an operad.
\begin{enumerate}
\item The linear map
\begin{equation}
\alpha_V: V\to \dsu(V), \quad \omega \longmapsto \svec{\omega}{\ast}, \quad \omega \in \genbas(n), n\geq 1, \svec{\omega}{\ast}:=\sum_{i\in [n]} \svec{\omega}{\{e_i\}}, \mlabel{eq:map1}
\end{equation}
induces a unique operad morphism $\alpha_\calp:\calp\longrightarrow \dsu(\calp)$ in the sense that $\alpha_\calp\circ p_V\circ i_V= p_{\dsu(V)}\circ i_{\dsu(V)}\circ \alpha_V$, that is, the following diagram commutes.
$$ \xymatrix{ V \ar^{\alpha_V}[d] \ar^{i_{V}}[rr] && \mathcal{T}(V)  \ar^{p_{V}}[rr] & & \calp
\ar^{\alpha_{\calp}}[d] \\
\dsu(V)\ar^{i_{\dsu(V)}}[rr] && \mathcal{T}(\dsu(V)) \ar^{p_{\dsu(V)}}[rr] & & \dsu(\calp)
}
$$
\mlabel{it:operad1}
\item
Let $\confm$ be a \conf with index $m$ and suppose $\max \{|\lin(r_{s})|\}_{s} \leq m$. Then the linear map
\begin{equation}
\gamma_V: V\to \ssp(V), \quad \omega \longmapsto \svec{\omega}{\star}, \quad \omega \in \genbas(n), n\geq 1, \svec{\omega}{\star}:=\sum_{\emptyset\neq I\subseteq [n]} \svec{\omega}{e_I},
\end{equation}
induces a unique operad morphism $\gamma_\calp:\calp\longrightarrow \ssp(\calp)$ in the above sense.
\mlabel{it:operad2}
\item Let $\confm$ be a \conf with
index $m$ and suppose $\max \{|\lin(r_{s})|\}_{s} \leq m$. Then the linear map
\begin{equation}\mlabel{eq:map3}
\omega \longmapsto \svec{\omega}{e_{[n]}}, \quad \omega \in \genbas(n), n\geq 1,
\end{equation}
induces a unique operad morphism from $\calp$ to $\ssp(\calp)$ in the above sense.
\mlabel{it:operad3}
\end{enumerate}
\end{theorem}

\begin{proof}
We assume that $R$ is given by Eq. (\ref{eq:pres}).

\noindent
(\ref{it:operad1})
By the universal property of the free operad $\mathcal{T}(V)$ on the $\BS$-module $V$, the $\BS$-module morphism
$i_{\dsu(V)}\circ \alpha_V : V \to \mathcal{T}(\dsu(V))$ induces a unique operad morphism
$ \free{\alpha}_V: \mathcal{T}(V)\to \mathcal{T}(\dsu(V))$
such that $\free{\alpha}_V \circ i_{V}= i_{\dsu(V)} \circ \alpha_V$.

By Lemma \ref{lem:all}, Eq. (\ref{eq:bsuall}) holds. Hence we have
$$
\sum_{i} c_{s,i} \dsu(\tau_{s,i})= \sum_{i} \sum_{x \in \lin(\tau_{s,i})} c_{s,i} \dsu_{x}(\tau_{s,i}), 1\leq s\leq k.
$$
Since $L_s:=\lin(\tau_{s,i})$ does not depend on $i$, we have
$$
\sum_{i} c_{s,i} \dsu(\tau_{s,i})=\sum_{x \in L_{s}} \dsu_{x} \left(\sum_{i}c_{s,i}\tau_{s,i}\right)=0,\quad
1\leq s\leq k.
$$
Therefore, $(\dsu(R))  \subseteq \ker (\free{\alpha}_V)$. Thus there is a unique operad morphism $\alpha_{\calp}:\calp =\mathcal{T}(V)/(R) \to \dsu(\calp):=\mathcal{T}(\dsu(V))/(\dsu(R))$ such that $\alpha_{\calp} \circ p_{V} = p_{\dsu(V)} \circ \free{\alpha}_V.$ We then have
$\alpha_{\calp} \circ p_V\circ i_{V} =p_{\dsu(V)}\circ i_{\dsu(V)} \circ \alpha_V$.

Suppose that $\alpha_{\calp}':\calp\to \dsu(\calp)$ is another operad morphism such that $p_{\dsu(V)} \circ i_{\dsu(V)} \circ \alpha_V = p_{\dsu(V)} \circ \free{\alpha}_V \circ i_{V}.$ By the universal property of the free operad $\mathcal{T}(V)$, we obtain
$\alpha_{\calp}' \circ p_{V} = p_{\dsu(V)} \circ \free{\alpha} = \alpha_{\calp} \circ p_{V}.$
Since $p_{V}$ is surjective, we obtain $\alpha_{\calp}'=\alpha_{\calp}.$ This proves the uniqueness of $\alpha_{\calp}.$

\smallskip

\noindent (\mref{it:operad2}) The proof is similar to the proof of Item~(\mref{it:operad1}). The linear map $\gamma:V\to \ssp(V)$ extends uniquely to $\free{\gamma}_V: \calt(V)\to \calt(\ssp(V))$.
By Corollary \mref{coro:special} we have
$$
\sum_{i} c_{s,i} \ssp(\tau_{s,i})= \sum_{i} \sum_{I \in \confn_{|\lin(r_{s})|}} c_{s,i} \ssp_{I}(\tau_{s,i}), 1\leq s\leq k.
$$
Since $L_s:=\lin(\tau_{s,i})$ does not depend on $i$, we have
$$
\sum_{i} c_{s,i} \ssp(\tau_{s,i})=\sum_{I \in L_{s}} \ssp_{I} \left(\sum_{i}c_{s,i}\tau_{s,i}\right)=0,\quad
1\leq s\leq k.
$$
Therefore $(\ssp(R))\subseteq \ker (\free{\gamma}_V)$ and then the rest of the proof follows.
\smallskip

\noindent (\mref{it:operad3})
It is easy to see that the linear map defined in Eq.~(\mref{eq:map3}) is $\BS_n$-equivariant. So it induces a morphism of operads from $\mathcal{T}(\gensp)$ to $\ssp(\calp)$. Moreover, by the definition of \spl-splitting, we have
$$
\sum_{i} c_{s,i}\ssp_{\lin(\tau_{s,i})}(\tau_{s,i})=0,\quad 1\leq s\leq
k.
$$
Note that the labeled tree $\ssp_{\lin(\tau_{s,i})}(\tau_{s,i})$ is obtained by replacing
the label of each vertex of $\tau_{s,i}$, say $\gop\in \genbas(n)$ by $\svec{\gop}{e_{[n]}}$. Hence the conclusion follows.
\end{proof}

\begin{coro}\mlabel{coro:power}
{\rm If the index of $\confm$ is $\infty$ (that is we take the power-splitting), then (\mref{it:operad2}) and (\mref{it:operad3}) hold for any operad.}
\end{coro}

If we take $\calp$ to be the operad of partially or totally associative $3$-algebras, $3$-Lie algebra or generalized Lie algebra of order $3$, then we obtain the following results:
\begin{coro}\label{coro:re}
\begin{enumerate}
\item
Let $(A,\nwarrow, \uparrow, \nearrow)$ be a partially (resp. totally) dendriform $3$-algebra. Then the operation
\begin{equation}
\ast:=\nwarrow + \uparrow + \nearrow
\mlabel{eq:partd}
\end{equation}
makes $A$ into a partially (resp. totally) associative $3$-algebra.
\item Let $(A, \{\cdot,\cdot,\cdot\})$ be a $3$-pre-Lie algebra (resp. generalized pre-Lie algebra of order $3$). Then the operation
\begin{equation}
[x,y,z] := \{x,y,z\} + \{y,z,x\} + \{z,x,y\}
\mlabel{eq:3prelielie}
\end{equation}
gives a $3$-Lie (resp. generalized Lie algebra of order $3$) structure on $A$. The commutator of $3$-pre-Lie algebra (resp. generalized pre-Lie algebra of order $3$) is a 3-Lie algebra (resp. generalized Lie algebra of order $3$) with the structure given by $2[x,y,z]$.
\end{enumerate}
\end{coro}

\subsection{Compatibility of splittings}
We prove the functorial property of taking splittings.
\begin{theorem}
Let $\eta : \calp \longrightarrow \calq$ be an operad morphism. Then $\eta$ induces an operad morphism $\ssp(\eta): \dsu(\calp)\longrightarrow \dsu(\calq)$ such that
\begin{equation}
\eta\circ \alpha_\calp = \dsu(\eta)\circ \alpha_\calq,
\mlabel{eq:func}
\end{equation}
for the maps $\alpha_\calp$ and $\alpha_\calq$ in Theorem~\mref{pp:suast}. \mlabel{thm:inducemor}
\end{theorem}
A similar statement holds for $\tsu(\eta):\tsu(\calp)\longrightarrow \tsu(\calq)$.

\begin{proof}
We take the generating $\BS$-modules $V, W$ of $\calp= \calt(V)/(R_{\calp})$ and $\calq= \calt(W)/(R_{\calq})$ to be the $\BS$-modules of $\calp$ and $\calq$ respectively.
The operad morphism $\eta$ defines a family of $\BS_{n}$-equivalent maps $\tilde{\eta}_{n}: V(n) = \calp(n) \longrightarrow W(n) = \calq(n)$. Define a chain map $\theta: \dsu(V)  \longrightarrow \dsu(W)$ by
\begin{eqnarray*}
\theta_{n}: \dsu(V)(n) &\longrightarrow& \dsu(W)(n)\\
\svec{\omega}{e_{I}}  &\longmapsto& \svec{\eta_{n}(\omega)}{e_{I}}, \quad \omega \in V(n), I \in \confn_{n}.
\end{eqnarray*}
We use the following diagram to keep track of the maps that we will use below.
\begin{eqnarray}\mlabel{diag:combi}
\begin{xy}
(0,15)*+{\dsu(V)}="v1";(50,15)*+{\calt(\dsu(V))}="v2";(100,15)*+{\dsu(\calp)}="v3"; (15,5)*+{V}="v4";(50,5)*+{\calt(V)}="v5";
(85,5)*+{\calp}="v6";(15,-8)*+{W}="v7";(50,-8)*+{\calt(W)}="v8";(85,-8)*+{\calq}="v9";(0,-18)*+{\dsu(W)}="v10";(50,-18)*+{\calt(\dsu(W))}="v11";
(100,-18)*+{\dsu(\calq)}="v12";
{\ar@{->}^{i_{\dsu(V)}}"v1";"v2"};{\ar@{->}^{p_{\dsu(V)}}"v2";"v3"};{\ar@{->}^{\alpha_{V}}"v4";"v1"};{\ar@{->}^{i_{V}}"v4";"v5"};
{\ar@{->}^{p_{V}}"v5";"v6"};{\ar@{->}^{\eta}"v4";"v7"};{\ar@{->}_{\,\free{\alpha}_{V}}"v5";"v2"};
{\ar@{->}^{i_{W}}"v7";"v8"};{\ar@{->}^{p_{W}}"v8";"v9"};{\ar@{->}^{\free{\eta}}"v5";"v8"};{\ar@{->}^{\eta}"v6";"v9"};{\ar@{->}^{\theta}"v1";"v10"};
{\ar@{->}^{i_{\dsu(W)}}"v10";"v11"};{\ar@{->}^{p_{\dsu(W)}}"v11";"v12"};
{\ar@{->}^{\dsu(\eta)}"v3";"v12"};{\ar@{->}_{\alpha_{\calp} }"v6";"v3"};{\ar@{->}^{\alpha_{\calq}\ \ }"v9";"v12"};
{\ar@{->}_{\alpha_{W}}"v7";"v10"};{\ar@/_3pc/^{\free{\theta}}"v2";"v11"};{\ar@{->}^{\free{\alpha}_{W}}"v8";"v11"};
\end{xy}
\end{eqnarray}
It follows from the fact that $\eta$ is an operad morphism and the universal property of $\calt(V)$ that $(\free{\eta} (R_{\calp}))  \subseteq (R_{\calq})$. On the other hand, for any tree $\tau\in \calt(V)$ and $J \in \confn_{|\lin(\tau)|}$, we have $\dsu_{J} (\free{\eta}(\tau)) = \free{\theta} (\dsu_{J}(\tau))$. Therefore, $(\free{\theta} (\dsu(R_{\calp}))) = (\dsu(\free{\eta}(R_{\calp}))) \subseteq (\dsu(R_{\calq}))$. Then there exists a morphism $\dsu(\eta): \dsu(\calp) \longrightarrow \dsu(\calq)$ such that $\dsu(\eta) \circ p_{\dsu(V)} = p_{\dsu(W)} \circ \free{\theta}$. Further, by the universal property of $\calt(V)$, we get the commutativity of the leftmost trapezoid. In summary, all quadrilaterals in the Diagram (\mref{diag:combi}) are commutative except the rightmost trapezoid which is precisely Eq.~(\mref{eq:func}). To prove it, by the surjectivity of $p_V$ and the universal property of $\calt(V)$, we only need to prove
$ \eta\circ \alpha_\calp\circ p_V\circ i_V = \dsu(\eta)\circ \alpha_\calq\circ p_V\circ i_V.$
This follows from a diagram chase by the commutativity of the other quadrilaterals of the Diagram~(\mref{diag:combi}).
\end{proof}

From the morphism in Remark~\mref{rem:paslie}, we obtain
\begin{coro}
We have the following commutative diagram.
\begin{equation}
\xymatrix{
\text{\rm Partially dendriform 3-algebra}
\ar[rr]^{Eq.~(\mref{eq:partd})} \ar[d] && \text{\rm Partially associative 3-algebra} \ar[d]^{Remark~\ref{rem:paslie}} \\
\text{\rm Generalized pre-Lie algebra of order 3} \ar[rr]^{Eq.~(\mref{eq:3prelielie})} && \text{\rm Generalized Lie algebra of order 3}
}
\label{eq:diagram}\end{equation}
where the left vertical map is defined by
\begin{equation}
\{x,y,z\} = \nwarrow(x,y,z) + \uparrow(x,y,z) + \nearrow(x,y,z) - (\nwarrow(x,z,y) + \uparrow(x,z,y) + \nearrow(x,z,y))
\notag \mlabel{eq:pdendprelie}
\end{equation}
for a partially dendriform $3$-algebra $(A,\nwarrow, \uparrow, \nearrow)$.
\end{coro}

\delete{
\begin{proof}
This follows from Corollary \ref{coro:re}, Proposition \ref{pp:pasprelie} and Remark~\mref{rem:paslie}.
\end{proof}
}

Diagram (\ref{eq:diagram}) can be regarded as a generalization in the context of $3$-algebras of the diagram \cite{Ch}.
\begin{equation}\notag 
\xymatrix{ \mbox{Dendriform algebra} \ar[rr] \ar[d] &&
 \mbox{Associative algebra} \ar[d] \\
  \mbox{Pre-Lie algebra} \ar[rr]
 &&
 \mbox{Lie algebra}
}\end{equation}

\begin{prop}\label{prop:inclusion}
Any $3$-pre-Lie algebra is a generalized pre-Lie algebra of order $3$.
\end{prop}

\begin{proof}
This follows from Theorem~\mref{thm:inducemor} since any $3$-Lie algebra is a generalized Lie algebra of order $3$.
\end{proof}

The following results relate different splittings of an operad together.
\begin{prop}\mlabel{pp:rests}
Let $\calp=\mathcal{T}(V)/(R)$ be an operad, $\confm$ and $\confm'$ be two $\BS$-invariant \confs such that $\confm_{n} \subseteq \confm_{n}'$ for each $n$. Then there is a morphism of operads from $\confm'\mathrm{Sp}(\calp)$ to $\confm\mathrm{Sp}(\calp)$ that extends the linear map defined by
\begin{equation}
\svec{\gop}{e_{I}}\mapsto \svec{\gop}{e_{I}},\quad  \svec{\gop}{e_{J}} \mapsto 0,\quad \gop \in \gensp(n), n \geq 2, I \in S'_{n} \cap \confn_{n}, J \in S'_{n} \backslash \confn_{n}. \mlabel{eq:sutsze}
\end{equation}
\end{prop}

\begin{proof}
Let $R$ be given by Eq.~(\ref{eq:pres}).
\smallskip
The linear map defined by Eq.~(\mref{eq:sutsze}) is $\BS_n$-equivariant. Hence it induces a morphism of operads $\varphi: \confm' \mathrm{Sp}(\calp)\rightarrow \confm \mathrm{Sp}(\calp)$, and $\varphi\left(\svec{\gop}{\star}\right)=\svec{\gop}{\ast}$, where $\svec{\gop}{\star} = \sum_{I \in S'_{|\omega|}} \svec{\gop}{e_{I}}$, $\svec{\gop}{\ast} = \sum_{I \in \confn_{|\omega|}} \svec{\gop}{e_{I}}$. Then, we have
$$\varphi(\confm' \mathrm{Sp}_{I}(\tau_{s,i}))=\confm \mathrm{Sp}_{I}(\tau_{s,i}) \ , \ \text{ for all } I \in I \in S'_{n} \cap \confn_{n}$$
and
$$ \varphi(\confm' \mathrm{Sp}_{J}(\tau_{s,i}))=0 \ , \ \text{ for all } J \in S'_{n} \backslash \confn_{n} \ .$$
\end{proof}

If we take $\calp$ to be the operad of partially associative $n$-algebras, then we obtain the following results:
\begin{coro}\label{coro:zero}
\begin{enumerate}
\item
Let $(A,\quad \begin{picture}(5,5)(0,0)
\vector(-1,1){10}
\end{picture}, \begin{picture}(5,5)(0,0)
\vector(0,1){10}
\end{picture}, \begin{picture}(5,5)(0,0)
\vector(1,1){10}
\end{picture}\,\,, \quad \begin{picture}(5,5)(0,0)
\vector(-1,1){10}\vector(0,1){10}
\end{picture},\,\,~~~\begin{picture}(5,5)(0,0)
\vector(-1,1){10}\vector(1,1){10}
\end{picture}\,\,, ~~~~\begin{picture}(5,5)(0,0)
\vector(0,1){10}\vector(1,1){10}
\end{picture}\,\,, \quad \begin{picture}(5,5)(0,0)
\vector(-1,1){10}\vector(0,1){10}\vector(1,1){10}
\end{picture}~~~\,\,)$ be a partially tridendriform $3$-algebra. If the operations $\quad \begin{picture}(5,5)(0,0)
\vector(-1,1){10}\vector(0,1){10}
\end{picture},~~~\,\,\begin{picture}(5,5)(0,0)
\vector(-1,1){10}\vector(1,1){10}
\end{picture}\,\,, ~~~~\begin{picture}(5,5)(0,0)
\vector(0,1){10}\vector(1,1){10}
\end{picture}\,\,, \quad \begin{picture}(5,5)(0,0)
\vector(-1,1){10}\vector(0,1){10}\vector(1,1){10}
\end{picture}~~~\,\,$ are trivial, then $(A,\quad \begin{picture}(5,5)(0,0)
\vector(-1,1){10}
\end{picture}, \begin{picture}(5,5)(0,0)
\vector(0,1){10}
\end{picture}, \begin{picture}(5,5)(0,0)
\vector(1,1){10}
\end{picture}~~~\,\,)$ becomes a partially dendriform $3$-algebra.
\item
Let  $(A,\quad \begin{picture}(5,5)(0,0)
\vector(-1,1){10}
\end{picture}, \begin{picture}(5,5)(0,0)
\vector(0,1){10}
\end{picture}, \begin{picture}(5,5)(0,0)
\vector(1,1){10}
\end{picture}~~~\,\,)$ be a partially dendriform $3$-algebra. Then  $(A,\quad \begin{picture}(5,5)(0,0)
\vector(-1,1){10}
\end{picture}, \begin{picture}(5,5)(0,0)
\vector(0,1){10}
\end{picture}, \begin{picture}(5,5)(0,0)
\vector(1,1){10}
\end{picture}~~~\,\,,0,0,0,0)$ carries a partially tridendriform $3$-algebra structure, where 0 denotes the trivial operation.
\end{enumerate}
\end{coro}

\section{Splittings of operads, Rota-Baxter operators on operads and relative Rota-Baxter operators}
\mlabel{sec:rb}

In this section we establish the relationship between
the \spl-splitting of an operad on one hand and the actions of a Rota-Baxter operator on the operad on the other.
For this purpose, we generalize the concept of a Rota-Baxter operator~\mcite{Ba,CK,Gub,R1} from binary operads to general operads. By generalizing the concept of a relative Rota-Baxter operator (previously called an $\calo$-operator) from the binary case to the general case, we further show that any \spl-splitting of an operad can be recovered, on the level of algebras for an operad, by a relative Rota-Baxter operator.

\subsection{Splittings and Rota-Baxter operators on operads}
We define the Rota-Baxter operator on an operad, together with a \conf. As preparation, we first consider it on the level of algebras.

\begin{defn}
{\rm
Let $n\geq 1$ and let $\confn$ be an $\BS_n$-invariant subset of $B_{n}$ (=$\{\emptyset \neq J\subseteq [n]\}$).
Let $(A,\langle, \cdots,\rangle)$ be an {\bf $n$}-algebra consisting of a module $A$ over a commutative ring $\bfk$ and an $n$-ary operation
$$
\langle , \cdots , \rangle :  A^{\ot n} \longrightarrow A.
$$
A {\bf $C$-Rota-Baxter operator of weight $\lambda$} on $(A,\langle , \cdots , \rangle)$ is a linear map $P: A \longrightarrow A$ such that
\begin{equation}
\langle P(x_{1}), \cdots, P(x_{n}) \rangle = P \left( \sum_{I \in \confn}  \lambda^{|I| -1}   \langle \bar{P}(x_{1}), \cdots , \bar{P}(x_{i}), \cdots, \bar{P}(x_{n}))  \rangle\right),
\end{equation}
where $\bar{P}(x_{i}) =\left\{ \begin{array}{ll} x_{i} & i \in I\\ P(x_{i}) &i \notin I \end{array} \right.$ for all $x_{1},x_{2},\cdots,x_{n} \in A$. Then $A$ is called a {\bf $C$-Rota-Baxter $n$-algebra of weight $\lambda$}.
}
\end{defn}

\begin{remark}
{\rm
For any $n$-ary algebra, when $\confn = A_{n}(=\{i\in [n]\})$, a $C$-Rota-Baxter operator is just an usual Rota-Baxter operator of weight zero;
when $\confn = B_{n}$, a $C$-Rota-Baxter operator is just the usual Rota-Baxter operator of weight $\lambda$ ~\mcite{BGLW}.
}
\end{remark}

We next consider the action of Rota-Baxter operators on the level of operads.
\begin{defn}\mlabel{defn:av}
{\rm Let $\gensp$ be an $\BS$-module with $V(1) = \bfk\,\id$ and $\confm$ be an $\BS$-invariant \conf.
\begin{enumerate}
\item
Let $\bvp$ denote the $\BS$-module with $\bvp(1)=\bfk\,P, P\neq \id,$ and $\bvp(n) = V(n), n \geq 2$, where $P$ is a symbol. Let $\mathcal{T}(\bvp)$ be the free operad generated by $\bvp$.
\item
Define an $\BS$-module $\ssp(V)$ with $\ssp(\gensp)(n)=\gensp(n) \ot \left(\bigoplus_{I\in \confn_{n}}   \bfk e_{I}\right)$ as in Eq.~(\mref{eq:dsp}). Let $P^{\ot n,I}$ denote the $n$-th tensor power of $P$ but with the component from $I$ replaced by the identity map.
\delete{
So, for example, for the three inputs $x_1, x_2$ and $x_{3}$ of $P^{\ot 3}$, we have $P^{\ot 3, \{x_1\}}= {\rm id} \ot P \ot P$ and $P^{\ot 3, \{x_1,x_2\}}={\rm id}\ot {\rm id} \ot P$.
}
Define a linear map of graded vector spaces from $\ssp(V)$ to $\bvp$ by the correspondence:
$$
\xi:\quad \svec{\gop}{e_{I}} \mapsto \omega \circ P^{\ot n,I},\quad  \text{ for all } \gop\in \gensp (n),\quad  I \in \confn_{n},
$$
where $\circ$ is the operadic composition. By the universal property of the free operad, $\xi$ induces a homomorphism of operads that we still denote by $\xi$:
$$\xi:\mathcal{T}(\ssp(V))\to \mathcal{T}(\bvp).$$
\item
Let $\calp=\mathcal{T}(\gensp)/(R_\calp)$ be an operad defined by the $\BS$-module $\gensp$ and
relations $R_\calp$.
Let
\begin{equation}
\srb_{\mathcal{P}}(n):= \left\{\gop \circ P^{\ot n} - \sum_{I \in \confn_{n}} P \circ \omega \circ P^{\ot n,I}~\Big|~ \gop \in \gensp(n)\right\}.
\mlabel{eq:rbop}
\end{equation}
Define the {\bf operad of \spl-Rota-Baxter $\calp$-algebras}
by
$$\srb(\calp):=\mathcal{T}(\bvp)/( {\mathrm R}_\calp,\srb_{\calp}).$$
\end{enumerate}
}
\end{defn}

We first prove a lemma relating \spl-splitting and \spl-Rota-Baxter operators of weight one.

\begin{lemma}
Let $\calp=\mathcal{T}(\gensp)/(R_\calp)$ be an operad and $\confm$ be an $\BS$-invariant \conf.
Let $\tau\in \calt(\gensp)$ with $\lin(\tau)=n$.
\begin{enumerate}
\item
We have
\begin{equation}
P \circ \xi (~\ssp(\tau)~) \equiv \tau \circ P^{\ot n} \mod ({\mathrm R}_\calp, \srb_\calp).
\mlabel{eq:tpxi}
\end{equation}
\mlabel{it:tpxi}
\item
For $ J\in \confn_{n}$, we have
\begin{equation}
\xi(\ssp_{J}(\tau)) \equiv \tau \circ (P^{\ot n, J}) \mod ({\mathrm R}_\calp, \srb_\calp)\,.
\mlabel{eq:txidu}
\end{equation}
\mlabel{it:txidu}
\end{enumerate}
\mlabel{lem:txidu}
\end{lemma}

\begin{proof}
(\mref{it:tpxi}).
We prove by induction on $|\lin(\tau)|\geq 1$. When $|\lin(\tau)|=1$, $\tau$ is the tree with one leaf standing for the identity map. Then we have $\xi(~ \ssp(\tau)~)= \tau$ and $P \circ \xi(~\ssp(\tau)~)= P = \tau \circ P$, as needed. Assume that the claim has been proved for $\tau$ with $|\lin(\tau)| \leq k$ and consider a $\tau$ with $|\lin(\tau)|=k+1$. Then from the decomposition $\tau= \omega (\tau_{1} \vee \tau_{2} \vee \cdots \vee \tau_{\ell})$, we have
$\ssp(\tau)= \svec{\gop}{\star}(\ssp(\tau_{1}) \vee \ssp(\tau_{2}) \vee \cdots \vee \ssp(\tau_{\ell}))$ where $\svec{\gop}{\star}=\sum\limits_{I\in \confn_\ell} \svec{\gop}{I}$. Thus we have
\begin{eqnarray*}
P \circ \xi(~\ssp(\tau)~) &=& P \circ \xi \left(\svec{\gop}{\star}(~\ssp(\tau_{1}) \vee \ssp(\tau_{2}) \vee \cdots \ssp(\tau_{\ell})~)\right) \\
&=& P \circ \xi \left(\sum_{I \in \confn_{\ell}} \svec{\gop}{e_{I}}   (~\ssp(\tau_{1}) \vee \ssp(\tau_{2}) \vee \cdots \vee \ssp(\tau_{\ell})~)\right) \\
&=& P \circ \omega \circ \sum_{I \in \confn_{\ell}} (P^{\ot n, I} \circ (\xi(~\ssp(\tau_{1})) \ot \xi(~\ssp(\tau_{2})) \ot \cdots \ot \xi(~\ssp(\tau_{\ell}))))\\
&\equiv&  \omega \circ ((P \circ \xi(~\ssp(\tau_{1})))\ot (P \circ \xi(~\ssp(\tau_{2}))) \ot \cdots \ot (P \circ \xi(~\ssp(\tau_{\ell}))))  \\
&& \mod ({\mathrm R}_\calp, \srb_\calp) \qquad \text{(by Eq.~(\mref{eq:rbop})}\\
&\equiv& \omega \circ ((\tau_{1} \circ P^{\otimes |\lin(\tau_{1})|})  \otimes (\tau_{2} \circ P^{\otimes |\lin(\tau_{2})|}) \ot \cdots \ot (\tau_{\ell} \circ P^{\otimes |\lin(\tau_{\ell})|})) \\
&& \qquad ~~ \text{(by induction hypothesis)}\\
&=& \omega \circ (\tau_{1} \otimes \tau_{2} \ot \cdots \ot \tau_{\ell}) \circ P^{\otimes (k+1)}\\
&=&\omega \circ (\tau_{1} \vee \tau_{2} \vee \cdots \tau_{\ell}) \circ P^{\otimes (k+1)}\\
&=& \tau \circ P^{\otimes (k+1)},
\end{eqnarray*}
completing the induction.

\noindent
(\mref{it:txidu}). We again prove by induction on $|\lin(\tau)|$. If $|\lin(\tau)|=1$, then $1$ is the only leaf label of $\tau$ and $\{1\} \in \confn_{1}$. Thus we have
$ \xi(\ssp_{\{1\}}(\tau)) = \id = \tau \circ (P^{\ot 1, \{1\}})$, as needed.
Assume that the claim has been proved for all $\tau$ with $|\lin(\tau)|\leq k$ and consider a $\tau$ with $|\lin(\tau)|=k+1$. Write $\tau= \omega (\tau_{1} \vee \tau_{2} \vee \cdots \vee \tau_{\ell})$. Let $J \in \confn_{k+1}$ and $
I =\{i_{1},i_{2},\cdots, i_{s}\} = \{i~|~ 1 \leq i  \leq \ell, J \cap \lin(\tau_{i}) \neq \emptyset \}.$ It follows from the definition of the \conf with index $m$ that $I \in \confn_{\ell}$ and $J \cap \lin(\tau_{i})  \in  \confn_{|\lin(\tau_{i})|}$. Then we have
{\allowdisplaybreaks
\begin{eqnarray*}
\xi(\ssp_{J}(\tau))&=& \xi\left(\ssvec{\gop}{e_{I}} (\vee_{i=1}^\ell \ssp_{J \cap \lin(\tau_{i})}(\tau_{i}) )\right)\\
&=& \gop \circ \left(
\bigoplus_{i=1}^\ell (\tau_{i} \circ P^{\ot |\lin(\tau_{i})|, J \cap \lin(\tau_{i}) }) \right) \\
&&  \quad \quad \quad \text{(by induction hypothesis and Item~(\mref{it:tpxi}))} \\
&\equiv & \gop \circ (\tau_{1} \vee \cdots \vee \tau_{\ell}) \circ P^{|\lin(\tau)|, J} \mod ( {\mathrm R}_\calp,\srb_\calp)\\
&=& \tau \circ P^{\ot (k+1),J}.
\end{eqnarray*}
}
This completes the induction.
\end{proof}

The next result establishes the link between Rota-Baxter operator and splitting that unifies the previous known results~\mcite{Ag2,AL,BBGN,BGN1,BLN,E,HNB,Le3,LNB,NB} in this direction.

\begin{theorem}\label{thm:triav}
\begin{enumerate}
\item Let $\calp$ be an operad and $\confm$ be an $\BS$-invariant \conf. There is a morphism of operads
$$
\ssp(\calp) \longrightarrow \srb(\mathcal{P}),
$$
which extends the map $\xi$ given in  Definition~\ref{defn:av}.
\item Let $A$ be a $\mathcal{P}$-algebra. Let $P:A\to A$ be a \spl-Rota-Baxter operator. Then the following operations make $A$ into a
$\ssp(\opd)$-algebra:
$$\svec{\omega}{e_{I}}(x_{1},x_{2},\cdots,x_{n}):= \omega \circ P^{\ot, I} (x_{1},x_{2},\cdots,x_{n}), \quad \text{ for all } x_{1},x_{2},\cdots, x_{n}\in A, \omega \in \mathcal{P}(n), I \in \confn_{n}.$$
\end{enumerate}
\mlabel{thm:biav}
\end{theorem}

\begin{proof}
The second statement is just the interpretation of the first statement on the level of algebras. So we just need to prove the first statement. Let $R_{\ssp(\calp)}$ be the relation space of $\ssp(\calp)$. By definition, the relations of $\ssp(\calp)$ are generated by $\ssp_J(r)$ for locally homogeneous $r=\sum_i c_i \tau_i \in R_\calp$ and $J \in \confn_{|\lin(\tau_{i})|}$.
By Eqs.(\mref{eq:tpxi}) and (\mref{eq:txidu}), we have
$$
\xi \left (\sum_{i} c_i \ssp_{J}(\tau_{i})\right) = \sum_i c_i \xi(\ssp_{J}(\tau_{i})) \equiv \sum_i c_i \tau_i \circ P^{\ot n,J}
 \equiv \left (\sum_i c_i \tau_i \right) \circ P^{\ot n, J} \mod ({\mathrm R}_\calp, \srb_\calp).
$$
Hence $\xi(R_{\ssp(\calp)}) \subseteq ({\mathrm R}_\calp, \srb_\calp)$ and $\xi$ induces a morphism of operads
$$
\bar{\xi}: \ssp(\calp) \longrightarrow \srb(\calp).
$$
This proves the first statement.
\end{proof}

When we take $\mathcal{P}$ to be the operad of the 3-associative algebra or 3-Lie algebra, we obtain the following results.

\begin{coro}
{\rm
\begin{enumerate}
\item
Let $(A, (\cdot, \cdot, \cdot), P)$ be a Rota-Baxter totally (resp. partially) associative $3$-algebra of weight zero. Define three new operations on $A$ by
$$
\nwarrow(x,y,z)= (x, P(y),P(z)), \quad \uparrow(x,y,z) = (P(x),y,P(z)), \quad \nearrow(x,y,z) = (P(x),P(y),z).
$$
Then ~$(A, (\cdot, \cdot, \cdot))$ is a totally (resp. partially) dendriform $3$-algebra.
\item
Let $(L, [\cdot, \cdot, \cdot], P)$ be a Rota-Baxter 3-Lie algebra (resp. generalized Lie algebra of order $3$) of weight zero. Define a new operation on $L$ by
$$
\{x, y, z\} = [x, P(y),P(z)].
$$
Then ~$(L, \{\cdot, \cdot, \cdot\})$ is a $3$-pre-Lie algebra (resp. generalized pre-Lie algebra of order $3$).
\end{enumerate}
\mlabel{coro:RB}
}
\end{coro}

\subsection{Splittings and relative Rota-Baxter operators}
\mlabel{ss:O}
We generalize the concepts of a module and a relative Rota-Baxter operator \cite{BGN}. For simplicity, we also assume that the weight of a relative \spl-Rota-Baxter operator is one. As remarked above, this still include the case of weight 0 with a suitable choice of the \conf $\confm$ (namely when $\confm=\cala$).

To motivate our general definition of modules for a $\calp$-algebra where $\calp$ is any operad, we recall that an $A$-bimodule $M$ for an associative algebra $A$ can be equivalently defined to be an abelian group $M$ together with two actions $\ell_1, \ell_2:A\ot M\to M$ of $A$ such that the binary operation $\cdot$ on $A\oplus M$ defined by
$$ (a,m)\cdot (b,n):=(ab,\ell_1(a)n+\ell_2(b)m), \quad a, b\in A, n, m\in M,$$
turns $A\oplus M$ into an associative algebra.

\begin{defn}\label{defn:totallymodule}
{\rm Let $\calp=\calt(V)/(R)$ be an operad defined by an $\BS$-module $\gensp=\bigoplus\limits_{n\geq 1} \gensp_n$ with basis $\genbas=\bigoplus\limits_{n\geq 1} \genbas_n$ and by relations $R$. Let $A$ be a $\calp$-algebra and $\confm$ be an $\BS$-invariant \conf.
\begin{enumerate}
\item
Let $U$ be a $\bfk$-module. For each $\omega \in \genbas$, denote the arity of $\omega$ by $|\omega|$. Suppose there are linear maps
$$
l_{I}^{\omega}:A^{\otimes (|\omega|-|I|)}\ot U^{\ot{|I|}} \rightarrow U,  I=\{i_{1} <\cdots < i_{t}\} \in \confn_{|\omega|},
$$
such that $A\oplus U$ is turned into a $\mathcal{P}$-algebra by defining
the operations $\widetilde {\omega}$ on $A\oplus U$ by:
\begin{eqnarray}
\lefteqn{\widetilde{\omega} ((x_1,u_1),\cdots,(x_{|\omega|},u_{|\omega|}))} \notag\\
&:=&\Big(\omega(x_1,\cdots, x_{|\omega|}),
\sum_{I \in \confn_{|\omega|}} l^{\omega}_{I} (x_1,\cdots, x_{i_{1}-1},x_{i_{1}+1},\cdots, x_{i_{t}-1}, x_{i_{t}+1},\cdots, x_{|\omega|})(u_{i_{1}},u_{i_{2}}, \cdots, u_{i_{t}})\Big), \label{eq:defnop}
\end{eqnarray}
for all $\omega \in \genbas, x_j \in A$ and $u_j \in U, 1\leq j\leq |\omega|$. Then $U= (U, \{l_{I}^{\omega} \,|\, \omega\in \genbas,  I \in \confn_{|\omega|}\} )$ is called a {\bf \spl-module for the $\calp$-algebra $A$} or simply an {\bf $A$-\spl-module}.

\item
Let $U= (U, \{l_{I}^{\omega} \,|\, \omega\in \genbas, I \in \confn_{|\omega|}\} )$ be an $A$-\spl-module.
A linear map $\alpha: U \longrightarrow A$ is called a {\bf relative \spl-Rota-Baxter operator} (of weight one) on the $A$-\spl-module $U$ if
\begin{eqnarray*}
\lefteqn{\omega(\alpha(u_{1}),\alpha(u_{2}),\cdots, \alpha(u_{|\omega|}))}\\
&=& \sum_{I \in \confn_{|\omega|}} \alpha \big(l^{\omega}_{I} (\alpha(u_1),\cdots, \alpha(u_{i_{1}-1}), \alpha(u_{i_{1}+1}),\cdots,
\alpha(u_{i_{t}-1}), \alpha(u_{i_{t}+1}),\cdots, \alpha(u_{|\omega|}))(u_{i_{1}},u_{i_{2}}, \cdots, u_{i_{t}}) \big),
\end{eqnarray*}
for all $\omega \in \genbas$ and $x_j \in A, u_j \in U$, where $I = \{i_{1},i_{2},\cdots, i_{t}\}$ with $i_{1} < i_{2} < \cdots < i_{t}$.
\end{enumerate}
}
\end{defn}
To simplify the notations we will use the following abbreviations. For $\bfk$-modules $X, Y$, linear operator $\alpha:X\to Y$, vectors $\vec{x}=(x_1,\cdots,x_n)\in X^n$, $\vec{y}=(y_1,\cdots,y_n)\in Y^n$ and $I=\{i_1<\cdots <i_t\}\subseteq [n]$, denote
\begin{eqnarray}
\vec{x}_I &: = & (x_1,\cdots, x_{i_1-1},x_{i_1+1},\cdots, x_{i_t-1},x_{i_t+1},\cdots,x_n), \notag\\
{}_I\vec{x} &:=& (x_{i_1},\cdots, x_{i_t}), \notag\\
\vec{x_I y}&:=& (x_1,\cdots,x_{i_1-1},y_{i_1},x_{i_1+1},\cdots, x_{i_t-1},y_{i_t},x_{i_t+1},\cdots,x_n),\mlabel{eq:alt}\\
(\vec{x},\vec{y})&:=&((x_1,y_1),\cdots, (x_n,y_n))\in (X\oplus Y)^n,\notag\\
(\vec{x},0)&:=& ((x_1,0),\cdots,(x_n,0))\in (X\oplus Y)^n,\notag\\
(\alpha(\vec{x}))&:=& (\alpha(x_1),\cdots,\alpha(x_n)).\notag
\end{eqnarray}
Thus in the above definition, we have
$$\widetilde{\omega}(\vec{x},\vec{u}) =\left(\omega(\vec{x}), \sum_{I\in C_{|\omega|}} l^{\omega}_I(\vec{x}_I)({}_I\vec{u})\right), \quad \omega(\alpha(\vec{u}))= \sum_{I\in C_{|\omega|}} \alpha(l^{\omega}_I(\alpha(\vec{u})_I) ({}_I\vec{u})).$$

\begin{exam}
{\rm With the same notation as in Definition \ref{defn:totallymodule}, a \spl-Rota-Baxter operator of weight one on $A$ is a relative \spl-Rota-Baxter operator of weight one on the \spl-module $(A, \{l_{I}^{\omega}\,|\, \omega\in \genbas, I \in \confn_{|\omega|} \})$, where
with the notation in Eq.~(\mref{eq:alt}), define
$$l^{\omega}_{I} (\vec{x}_I)({}_I\vec{u}):= \omega(\vec{x_I\, u})
$$
for all  $\omega \in \genbas$, $x_{j} \in A$, $1 \leq j \leq |\omega|$,  $ I \in \confn_{|\omega|}$ and $I = \{i_{1},\cdots, i_{t}\}$ with $i_{1} < i_{2} < \cdots < i_{t}$.
}
\end{exam}

\begin{prop}
Let $A$ be a $\calp$-algebra, $\confm$ be an $\BS$-invariant \conf and $U= (U, \{l_{I}^{\omega} \,|\, \omega\in \genbas,  I \in \confn_{|\omega|}\} )$ be an $A$-\spl-module.
\begin{enumerate}
\item If $\confm=\cala$ and take $\ell^\omega_I=0$ when $|I|>1$, then $U$ is the usual module in the context of general $\calp$-algebra \mcite{MV,Sch}.
\mlabel{it:moda}
\item If $\confm=\calb$, then $U$ has a $\calp$-algebra structure.
\mlabel{it:modb}
\end{enumerate}
\mlabel{pp:mod}
\end{prop}

\begin{proof}
(\mref{it:moda}) is clear from the definition.
\smallskip

\noindent
(\mref{it:modb})
By the definition of an $A$-\spl-module and let $x_{1} = x_{2} = \cdots = x_{|\omega|} =0$ in Eq.(\ref{eq:defnop}), the operations $\omega_{U}(\vec{u}) := l_{[|\omega|]}^{\omega}(1_{\bfk})(\vec{u})$ make $U$ into a $\calp$-algebra.
\end{proof}

\begin{lemma}\label{lem:relative}
Let $(U, \{l_{I}^{\omega}\,|\, \omega\in \genbas, I \in \confn_{|\omega|}\})$ be an $A$-\spl-module, $\alpha: U \longrightarrow A$ be a linear map. Define $\alpha':  A \oplus U  \longrightarrow A \oplus U$ by
\begin{equation}
\alpha' (x,u): = (\alpha(u),0).
\mlabel{eq:a'}
\end{equation}
Then $\alpha$ is a relative \spl-Rota-Baxter operator of weight one if and only if $\alpha'$ is a \spl-Rota-Baxter operator of weight one on the $\calp$-algebra $A \oplus U$.
\end{lemma}
\begin{proof}
Since $(U, \{l_{I}^{\omega}\,|\, \omega\in \genbas, I \in \confn_{|\omega|}\})$ is an $A$-\spl-module, by definition, $A \oplus U $ has a $\calp$-algebra structure by the operations $\widetilde{\omega}, \omega\in \genbas$.

$\alpha:U\to A$ is a \spl-relative Rota-Baxter operator of weight one means
$$\omega(\alpha(\vec u))
= \sum_{I \in \confn_{|\omega|}} \alpha \big( l^{\omega}_{I} \big(\alpha(\vec{u})_I \big) ({}_I\vec{u}) \big).
$$
while $\alpha':A\oplus U\to A\oplus U$ is a \spl-Rota-Baxter operator of weight one means

\begin{eqnarray*}
\widetilde{\omega}(\alpha'(\vec x,\vec u))
&=& \alpha' \left( \sum_{I\in \confn_{|\omega|}} \widetilde{\omega}(\alpha'(\vec x,\vec u)_I\,(\vec x,\vec u))\right)= \alpha' \left( \sum_{I\in \confn_{|\omega|}}
\widetilde{\omega}((\alpha(\vec u),0)_I\,(\vec x, \vec u))\right) \\
&=& \alpha' \left(\sum_{I\in \confn_{|\omega|}}
(\widetilde{\omega}(\alpha(\vec u)_I\, \vec x, 0_I\,\vec u))\right)
=\alpha'\left(\sum_{I\in \confn_{|\omega|}}
\omega(\alpha(\vec u)_I\, \vec x), \sum_{J\in \confn_{|\omega|}}l_J^\omega((\alpha(\vec u)_I,\vec x)_J)({}_J\,(0_I\,\vec u))\right)\\
&=& \left(\sum_{I\in \confn_{|\omega|}}
\sum_{J\in \confn_{|\omega|}}
\alpha \left(l_J^\omega((\alpha(\vec u)_I,\vec x)_J)({}_J\,(0_I\,\vec u))\right), 0\right).
\end{eqnarray*}
\delete{
\begin{eqnarray*}
\widetilde{\omega}(\alpha'(\vec x,\vec u))
&=& \alpha' \left( \sum_{I\in \confn_{|\omega|}} \widetilde{\omega}(\alpha'(\vec x,\vec u)_I\,(\vec x,\vec u))\right)= \alpha' \left( \sum_{I\in \confn_{|\omega|}}
\widetilde{\omega}((\alpha(\vec x),0)_I\,(\vec x, \vec u))\right) \\
&=& \alpha' \left(\sum_{I\in \confn_{|\omega|}}
(\widetilde{\omega}(\alpha(\vec x)_I\, \vec x, 0_I\,\vec u))\right)
=\alpha'\left(\sum_{I\in \confn_{|\omega|}}
\omega(\alpha(\vec x)_I\, \vec x), \sum_{J\in \confn_{|\omega|}}l_J^\omega((\alpha(\vec x)_I,\vec x)_J)({}_J\,(0_I\,\vec u))\right)\\
&=& \left(\sum_{I\in \confn_{|\omega|}}
\sum_{J\in \confn_{|\omega|}}
\alpha \left(l_J^\omega((\alpha(\vec x)_I,\vec x)_J)({}_J\,(0_I\,\vec u))\right), 0\right).
\end{eqnarray*}
}

If $I \neq J$, there exists $0$-tuples in the vector ${}_J\,(0_{I}\, \vec u)$, since $l_{J}^{\omega}$ is multilinear, then we have
$$
l_{J}^{\omega} ((\alpha(\vec u)_I,\vec x)_J)({}_J\,(0_I\,\vec u))=0
$$
and it is easy to see that
$$ \sum_{I\in \confn_{|\omega|}}
\sum_{J\in \confn_{|\omega|}}
\alpha \left(l_J^\omega((\alpha(\vec u)_I,\vec x)_J)({}_J\,(0_I\,\vec u))\right) =
\sum_{I \in \confn_{|\omega|}} l^{\omega}_{I} \alpha\big(\alpha(\vec{u})_I) ({}_I\vec{u}).
$$

\delete{\begin{eqnarray*}
&&\alpha' \big( \widetilde{\omega} (\alpha'(x_{1},u_{1}),\cdots, \alpha'(x_{i_{1}-1}, u_{i_{1}-1}), (x_{i_{1}}, u_{i_{1}}), \alpha'(x_{i_{1}+1}, u_{i_{1}+1}),\cdots, \\
&&\quad \quad \alpha'(x_{i_{t}-1 }, u_{i_{t}-1}),(x_{i_{t}},u_{i_{t}}), \alpha'(x_{i_{t}+1},u_{i_{t}+1}),\cdots, \alpha'(x_{|\omega|}, u_{|\omega|}))\big)\\
&&=\alpha' \big( \widetilde{\omega} ((\alpha(u_{1}),0),\cdots, (\alpha(u_{i_{1}-1}),0), (x_{i_{1}}, u_{i_{1}}),(\alpha(u_{i_{1}+1}),0),\cdots, \\
&&\quad \quad (\alpha(u_{i_{t}-1}),0),(x_{i_{t}},u_{i_{t}}), (\alpha(u_{i_{t}+1}),0),\cdots, (\alpha(u_{|\omega|}),0))\big)\\
&&=\alpha'\Big(\omega(\alpha(u_{1}),\cdots, \alpha(u_{i_{1}-1}), u_{i_{1}}, \alpha(u_{i_{1}+1}),\cdots, \alpha(u_{i_{t}-1}), u_{i_{t}}, \alpha(u_{i_{t}+1}),\cdots, \alpha(u_{|\omega|})), \\
&& \quad \quad l^{\omega}_{I} \big(\alpha(u_1),\cdots, \alpha(u_{i_{1}-1}),\alpha(u_{i_{1}+1}),\cdots, \alpha(u_{i_{t}-1}), \alpha(u_{i_{t}+1}), \cdots,\alpha(u_{|\omega|}) \big) (u_{i_{1}},u_{i_{2}}, \cdots, u_{i_{t}}) \Big)\\
&&= \big( \alpha ( l^{\omega}_{I} \big(\alpha(u_1),\cdots, \alpha(u_{i_{1}-1}),\alpha(u_{i_{1}+1}),\cdots, \alpha(u_{i_{t}-1}), \alpha(u_{i_{t}+1}), \cdots,\alpha(u_{|\omega|}) \big) (u_{i_{1}},u_{i_{2}}, \cdots, u_{i_{t}})  ),0 \big).
\end{eqnarray*}
}
Hence $\alpha$ is a relative \spl-Rota-Baxter operator of weight one on module $U$ if and only if $\alpha'$ is a \spl-Rota-Baxter operator of weight one on $A \oplus U$.
\end{proof}

We have the following generalization of Theorem~\mref{thm:triav}.

\begin{theorem}\label{thm:ooperator}
 Let $\calp=\calt(V)/(R)$ be an operad defined by an $\BS$-module $\gensp=\bigoplus\limits_{n\geq 1} \gensp_n$ with basis $\genbas=\bigoplus\limits_{\geq 1} \genbas_n$ and by relations $R$. Let $\confm$ be an $\BS$-invariant \conf. Let $A$ be a $\calp$-algebra and $(U, \{l_{I}^{\omega}\,|\, \omega\in \genbas, I \in \confn_{|\omega|}\})$ be an $A$-\spl-module.
Let $\alpha: U \longrightarrow A$ be a relative \spl-Rota-Baxter operator of weight one on the $A$-\spl-module $U$.
For $\omega\in \genbas$ and $ I \in \confn_{|\omega|}$, define
\begin{equation}
\svec{\omega}{e_{I}}(\vec u)
:= l^{\omega}_{I} \big(\alpha(\vec{u})_I)({}_I\vec{u}), \label{eq:opweight}
\end{equation}
for $ u_{i} \in U$, $I \in \confn_{|\omega|}$ and $I = \{i_{1},\cdots, i_{t}\}$ with $i_{1} < i_{2} < \cdots < i_{t}$. Then
$$
\left(U, \left\{\svec{\omega}{e_{I}}\,\Big|\, \omega\in \genbas, I \in \confn_{|\omega|} \right\}\right)
$$
is a $\ssp(\opd)$-algebra.
Moreover, when the splitting is arity-splitting or power-splitting, there is a $\mathcal P$-algebra structure on $U$ given by
\begin{equation}
\svec{\omega}{\star}=\sum_{I \in \confn_{|\omega|}} \svec{\omega}{e_{I}},\;\; \omega \in \genbas,
\notag
\end{equation}
and $\alpha$ is a homomorphism of $\mathcal P$-algebras.
\end{theorem}

\begin{proof}
Let $\alpha$ be a relative \spl-Rota-Baxter operator of weight one. By Lemma \ref{lem:relative}, $\alpha'$ in Eq.~(\mref{eq:a'}) is a Rota-Baxter operator of weight one on $A \oplus U$. It follows from Theorem \ref{thm:triav} that the operations
$$
\svec{\widetilde{\omega}}{e_{I}}((\vec x, \vec u)):= \widetilde{\omega} \circ \alpha'^{\ot |\omega|, I} ((\vec x, \vec u)), \quad \omega\in \genbas, I\in \confn_{|\omega|}
$$
make $A \oplus U$ into a $\ssp(\opd)$-algebra. Furthermore,  we have
$$ \widetilde{\omega}\circ \alpha'^{\ot |\omega|,I} ((\vec x,\vec u))
=\widetilde{\omega}((\alpha(\vec u),0)_I)({}_I\,(\vec x,\vec u))
= (\omega(\alpha(\vec u)_I\,\vec x), \svec{\omega}{e_I}(\vec u)).$$
\delete{\begin{eqnarray*}
&&\widetilde{\omega} \circ \alpha'^{\ot |\omega|, I} ((x_{1},u_{1}),\cdots,(x_{n},u_{n}))\\
&& = \widetilde{\omega} \big((\alpha(u_{1}),0),\cdots, (\alpha(u_{i_{1}-1}),0), (x_{i_{1}}, u_{i_{1}}),(\alpha(u_{i_{1}+1}),0),\cdots, \\
&&\quad \quad  \quad (\alpha(u_{i_{t}-1}),0),(x_{i_{t}},u_{i_{t}}), (\alpha(u_{i_{t}+1}),0),\cdots, (\alpha(u_{|\omega|}),0)\big)\\
&& = ( \omega(\alpha(u_1),\cdots, \alpha(u_{i_{1}-1}),x_{i_{1}},\alpha(u_{i_{1}+1}),\cdots, \alpha(u_{i_{t}-1}), x_{i_{t}},\alpha(u_{i_{t}+1}), \cdots, \alpha(u_{|\omega|})) , \svec{\omega}{e_{I}}(u_{1}, \cdots, u_{|\omega|})).
\end{eqnarray*}
}
Obviously, $\overline{U}:= \{(0, v)~|~v \in U \}$ is a sub-$\ssp(\calp)$-algebra of of $A\oplus U$. By transporting of structures, we obtain an $\ssp(\calp)$-algebra structure on $U$. This is precisely the one defined in Eq.~(\ref{eq:opweight}). The last statement of the theorem follows from a direct computation using Proposition~\mref{pp:suast}.(\mref{it:operad1}).
\end{proof}

The following result gives an inverse of Theorem~\mref{thm:triav}, in the sense that any $\dsu(\calp)$ or $\tsu(\calp)$-algebra can be derived from a relative \spl-Rota-Baxter operator of weight one. See~\mcite{BGN,Uc} for the case of dendriform algebra and tridendriform algebra.
\begin{theorem}
Let $\calp=\calf(V)/(R)$ be as defined in Theorem~\mref{thm:ooperator} and $\confm$ be the \conf with index $1$ or $\infty$. Let $A$ be a given $\ssp(\mathcal{P})$-algebra with operations
$\left\{\svec{\omega}{e_{I}}\,\Big|\, \omega\in \genbas, I \in \confn_{|\omega|} \right\}$. For all $\vec{x}\in A^{|\omega|}, I= \{i_{1}<\cdots <i_{t}\}\in \confn_{|\omega|}$, define
$$
l^{\omega}_{I} ((\vec{x})_I)({}_I\,\vec{x}) = \svec{\omega}{e_I}(\vec{x}).
$$
Then $(A, \{l_{I}^{\omega}\,|\,  \omega\in \genbas, I \in \confn_{|\omega|}\})$ is a $\confm$-module for the $\mathcal P$-algebra $(A,\genbas_{\star})$, where
$$
\genbas_{\star} = \Big\{\,\svec{\omega}{\star}:=\sum_{I \in \confn_{|\omega|}}\svec{\omega}{e_{I}}\,\big|\, \omega \in \genbas \Big\}.
$$
Further, the identity linear map $\id:A\to A$ is a relative Rota-Baxter operator of weight one for the $\mathcal P$-algebra $(A,\genbas_{\star})$ associated to the $\confm$-module $(A, \{l_{I}^{\omega}\,|\,  \omega\in \genbas, I \in \confn_{|\omega|}\})$. Finally the $\ssp(\calp)$-algebra from the $\calp$-algebra $(A,\genbas_\star)$ obtained from the relative Rota-Baxter operator $\id$ by Theorem~\mref{thm:ooperator} is precisely $\left(A,\left\{\svec{\omega}{e_{I}}\,\Big|\, \omega\in \genbas, I \in \confn_{|\omega|} \right\}\right)$.
\label{thm:suoop}
\end{theorem}

\begin{proof}
By Proposition \ref{pp:suast}, $(A, \genbas_{\star})$ is a $\calp$-algebra.
For $\omega\in \genbas$, define an operation on $A \oplus A$ by
$$
\widetilde {\svec{\omega}{\star}} (\vec x ,\vec u)
: = \Big(\svec{\omega}{\star}(\vec x), \sum_{I \in \confn_{|\omega|}} \svec{\omega}{e_{I}}(\vec{x}_I\, \vec{u}))\Big).
$$
For $\tau \in \bigcup_{t\in \mathfrak{R}} t(\genbas)$, let $\overline{\tau}$ denote the redecoration of $\tau$ with each vertex $\omega$ of $\tau$ being replaced by $\svec{\omega}{\star}$. Also let $\widetilde{\tau}$ denote the redecoration of $\tau$ with each vertex $\omega$ of $\tau$ being replaced by $\widetilde{\svec{\omega}{\star}}$. Let $\lin(\tau) = [n]$. We claim that
\begin{equation}
 \widetilde{\tau}(\vec x, \vec u)
=\left(\overline{\tau}(\vec x),\sum_{I \in \confn_{|\lin(\tau)|}} \ssp_{I} (\tau)(\vec x_I\, \vec u)\right). \label{eq:moduleweight}
\end{equation}

We prove Eq.(\ref{eq:moduleweight}) by induction on $|\lin(\tau)|\geq 1$. When $|\lin(\tau)|=1$, $\tau$ is the tree with one leaf standing for the identity map. Then we have $\widetilde{\tau} (x_{1} + u_{1}) = x_{1} + u_{1} = \overline{\tau}(x_{1}) + \ssp_{\{1\}}(\tau)(u_{1})$. Assume that the claim has been proved for $\tau$ with $|\lin(\tau)| \leq n-1$ where $n\geq 2$ and consider a $\tau$ with $|\lin(\tau)|= n$. In the decomposition $\tau= \omega (\tau_{1} \vee \tau_{2} \vee \cdots \vee \tau_{\ell})$, denote the corolla with $\ell$ leaves by $T_{\ell}$.  Let $\lin(\tau_{p}) = \{k_{p-1}+1, \cdots, k_{p}\}$ and $\overrightarrow{x^{p}} = (x_{k_{p-1}+1},x_{k_{p-1}+2},\cdots,x_{k_{p}})$ with the convention that $k_{0} =0$ and $k_{\ell} =n+1$.

Let $H^{(p,i)}$ be any element of $\confn_{|\lin(\tau_{p})|}$, $\overrightarrow{u^{p}} =(u_{k_{p-1}+1},u_{k_{p-1}+2},\cdots, u_{k_{p}})$. Then
$$
\overrightarrow{ x^{p}}_{H^{(p,i)}}\overrightarrow{u^{p}}=\big(x_{k_{p-1}+1},\cdots,x_{h_{i1}-1}, u_{h_{i1}}, x_{h_{i1}+1},\cdots, x_{h_{i a_{i}}-1}, u_{h_{i a_{i}}} , x_{h_{i a_{i}}+1}, \cdots, x_{k_{p}}\Big).
$$
Denote $\mathfrak{U}_{p} = \displaystyle \sum_{ H^{(p,i)} \in \confn_{|\lin(\tau_{p})|}} \ssp_{H^{(p,i)}}(\tau_{p}) \overrightarrow{x^{p}}_{H^{(p,i)}}\overrightarrow{u^{p}}$.
For any $I \in \confn_{n}$, there exists some $q$ such that $I$ is of the form $I=I_{j_{1}}\sqcup \cdots \sqcup I_{j_{q}}$ with
$I_{j_{b}}:=I\cap \lin(\tau_{j_b})=\{h_{b1}< \cdots< h_{ba_b}\}, 1\leq b\leq q$. By the definition of \conf, there exist $H^{(j_{b},i_{0})} \in \confn_{|\lin(\tau_{j_{b}})|}$ such that ${}_{I_{j_{b}}}\vec u = {}_{H^{(j_{b},i_{0})}} \overrightarrow{u^{j_{b}}}$. Conversely, for any choice of $J \subseteq \confn_{\ell}$ with $J = \{j_{1}<j_{2}<\cdots< j_{q}\}$ and $H^{(j_{b},i_{0})} \in \confn_{|\lin(\tau_{j_{b}})|}$, there exists $\emptyset \neq I_{j_{b}} \subseteq \lin(\tau_{j_{b}}), 1\leq b\leq q$ such that ${}_{I_{j_{b}}}\vec u = {}_{H^{(j_{b},i_{0})}} \overrightarrow{u^{j_{b}}}$, we obtain $\emptyset \neq I:= I_{j_{1}} \sqcup \cdots \sqcup I_{j_{q}}\in \confn_{n}$.

Then we have
\begin{eqnarray*}
\widetilde{\tau} (\vec x, \vec u) &=& \widetilde{\svec{\omega}{\star}}\left(\widetilde{\tau}_{1}(\overrightarrow{x^{1}},\overrightarrow{u^{1}}),\cdots, \widetilde{\tau}_{\ell}(\overrightarrow{x^{\ell}},\overrightarrow{u^{\ell}})\right) \quad (\mbox{by the definition of}~\tilde{\tau})\\
&=& \widetilde{\svec{\omega}{\star}} \Big( (\overline{\tau}_{1}(\overrightarrow{x^{1}}), \mathfrak{U}_{1}),  \cdots, (\overline{\tau}_{\ell}(\overrightarrow{x^{\ell}}), \mathfrak{U}_{\ell}) \big) \quad (\mbox{by the induction hypothesis})  \\
&=& \Big(\svec{\omega}{\star} \big(\overline{\tau}_{1}(\overrightarrow{x^{1}}),\cdots, \overline{\tau}_{\ell}(\overrightarrow{x^{\ell}})\big), \displaystyle \sum_{J \in \confn_{\ell}} \svec{\omega}{e_{J}}\Big(\overline{\tau}_{1}(\overrightarrow{x^{1}}), \cdots,  \overline{\tau}_{j_{1}-1}(\overrightarrow{x^{j_{1}-1}}), \mathfrak{U}_{j_{1}}, \overline{\tau}_{j_{1}+1}(\overrightarrow{x^{k_{j_{1}}+1}}),\\
&&\qquad \cdots, \overline{\tau}_{j_{q}-1}(\overrightarrow{x^{k_{j_{q}-1}}}),\mathfrak{U}_{j_{q}}, \overline{\tau}_{j_{q}+1}(\overrightarrow{x^{k_{j_{q}}+1}}),  \cdots, \overline{\tau}_{\ell}(\overrightarrow{x^{\ell}})\Big)\Big) \quad\left(\mbox{by the definition of}~ \widetilde{\svec{\omega}{\star}}\right)\\
&=&\Big(\svec{\omega}{\star} \big(\overline{\tau}_{1}(\overrightarrow{x^{1}}),\cdots, \overline{\tau}_{\ell}(\overrightarrow{x^{\ell}})\big), \displaystyle \sum_{J \in \confn_{\ell}} \ssp_{J}(T_{\ell})\Big(\overline{\tau}_{1}(\overrightarrow{x^{1}}), \cdots,  \overline{\tau}_{j_{1}-1}(\overrightarrow{x^{j_{1}-1}}), \mathfrak{U}_{j_{1}}, \overline{\tau}_{j_{1}+1}(\overrightarrow{x^{k_{j_{1}}+1}}),\\
&&\qquad \cdots, \overline{\tau}_{j_{q}-1}(\overrightarrow{x^{k_{j_{q}-1}}}),\mathfrak{U}_{j_{q}}, \overline{\tau}_{j_{q}+1}(\overrightarrow{x^{k_{j_{q}}+1}}),  \cdots, \overline{\tau}_{\ell}(\overrightarrow{x^{\ell}})\Big)\Big)\\
&=& \left(\overline{\tau}(\vec x), \sum_{I \in \confn_{|\lin(\tau)|}} \ssp_{I}(\tau) (\vec x_{I} \vec u )\right).
\end{eqnarray*}

Let $\calp=\calf(V)(R)$ with $R$ given by Eq. (\ref{eq:pres}), that is
\begin{equation*}
r_s:=\sum_i c_{s,i}\tau_{s,i}, \ c_{s,i}\in\bfk, \tau_{s,i} \in \bigcup_{t\in \mathfrak{R}} t(\genbas), \ 1\leq s\leq k.
\end{equation*}
Recall that $A$ is a $\ssp(\opd)$-algebra with the operations
$\left\{\svec{\omega}{e_{I}} \,\Big|\, \omega\in \genbas,  I \in \confn_{|\omega|}\right\}$ and $(A, \genbas_{\star})$ is a $\calp$-algebra.
Denote $\widetilde{\genbas}_{\star}: = \left\{\widetilde{\svec{\omega}{\star}}~|~ \omega \in \genbas\right\}$.

For a given $1\leq s\leq k$, by Lemma \mref{lem:arity}, Corollary \mref{coro:power} and the definition of $\widetilde{\svec{\omega}{\star}}$, we have

$$\widetilde{r}_{s}(\vec x, \vec u)
= \sum_{i} c_{s,i} \big(\widetilde{\tau}_{s,i}(\vec x, \vec u)\big)
= \sum_{i}c_{s,i} \left( \overline{\tau}_{s,i}(\vec x),  \sum_{I \in \confn_{n}} \ssp_{I}(\tau_{s,i})(\vec x_I\, \vec u) \right)
=\left (\overline{r}_{s}(\vec x), \sum_{I \in \confn_{n}} \ssp_{I} (r_{s}) (\vec x_I \vec u)\right)
$$
which is $(0,0)$ since $(A, \genbas_{\star})$ is a $\calp$-algebra and $A$ is a $\ssp(\opd)$-algebra.
Thus $(A \oplus A, \widetilde{\genbas}_{\star})$ is a $\calp$-algebra. Hence $(A, \{l_{I}^{\omega} \,|\, \omega\in \genbas, I \in \confn_{|\omega|}\})$ is an $A$-\spl-module for the $\mathcal P$-algebra $(A,\genbas_{\star})$.

Further the linear map $\id: (A,\{l_{I}^{\omega} \,|\, \omega\in \genbas, I \in \confn_{|\omega|}\})\to (A,\genbas_{\star})$ is a relative \spl-Rota-Baxter operator of weight one since
$$
\svec{\omega}{\star} (\id(\vec x))
= \svec{\omega}{\star}(\vec x)
= \sum_{I \in \confn_{|\omega|}} \svec{\omega}{e_{I}}(\vec x)
=\id \Big(\sum_{I \in \confn_{|\omega|}} l^{\omega}_{I} \big(\id(\vec x_I))({}_I\,\vec x) \Big).
$$

The last statement of the theorem follows from the definition of $l_I^{\omega}$ in the theorem.
\end{proof}

\smallskip

\noindent {\bf Acknowledgements: } C. Bai would like to thank the support by NSFC (11271202, 11221091) and SRFDP (20120031110022).
L. Guo acknowledges support from NSF grant DMS 1001855.

\end{document}